\documentclass{amsart}

\usepackage[utf8]{inputenc}

\usepackage[dvipsnames]{xcolor}
\usepackage[colorlinks]{hyperref}
\hypersetup{
linkcolor=BrickRed
,citecolor=Green
,filecolor=Mulberry
,urlcolor=NavyBlue
,menucolor=BrickRed
,runcolor=Mulberry
,linkbordercolor=BrickRed
,citebordercolor=Green
,filebordercolor=Mulberry
,urlbordercolor=NavyBlue
,menubordercolor=BrickRed
,runbordercolor=Mulberry
}

\usepackage{graphicx}
\usepackage{mathtools}
\mathtoolsset{showonlyrefs}
\usepackage[foot]{amsaddr}

\usepackage{bm}

\usepackage[numbers]{natbib}
\newtheorem{theorem}{Theorem}
\newtheorem{proposition}[theorem]{Proposition}

\newtheorem{lemma}[theorem]{Lemma}

\theoremstyle{definition}

\theoremstyle{remark}
\newtheorem{remark}{Remark}[section]

\newcommand{\abs}[1]{\left\lvert #1 \right\rvert}
\newcommand{\absbig}[1]{\Big\lvert #1 \Big\rvert}
\newcommand{\norm}[1]{\left\lVert #1 \right\rVert}
\newcommand{\an}[1]{\ensuremath{\left\langle#1\right\rangle}} %angles
\newcommand{\anbig}[1]{\ensuremath{\Big\langle#1\Big\rangle}} %angles
\newcommand{\sbra}[1]{\ensuremath{\left[#1\right]}} %square brackets []
\newcommand{\pbra}[1]{\ensuremath{\left\{#1\right\}}} %Poisson bracket {}
\newcommand{\parenthese}[1]{\ensuremath{\left(#1\right)}} %parentheses

\definecolor{codegreen}{rgb}{0,0.6,0}
\definecolor{codegray}{rgb}{0.5,0.5,0.5}
\definecolor{codepurple}{rgb}{0.58,0,0.82}
\definecolor{backcolour}{rgb}{0.95,0.95,0.92}
\definecolor{olive}{RGB}{186,184,108}

\definecolor{fore}{RGB}{249,242,215}
\definecolor{back}{RGB}{51,51,51}
\definecolor{title}{RGB}{255,0,90}
\definecolor{dgreen}{rgb}{0.,0.6,0.}
\definecolor{gold}{rgb}{1.,0.84,0.}
\definecolor{JungleGreen}{cmyk}{0.99,0,0.52,0}
\definecolor{BlueGreen}{cmyk}{0.85,0,0.33,0}
\definecolor{RawSienna}{cmyk}{0,0.72,1,0.45}
\definecolor{Magenta}{RGB}{255, 3, 100}
\definecolor{blackgrey}{rgb}{0.16,0.21,0.26}
\definecolor{mycolor}{RGB}{252,186,3}

\newcommand{\arxivonly}[1]{#1}
\newcommand{\armaonly}[1]{}

\newcommand{\C}{\mathbb{C}}
\newcommand{\CS}{C^\infty(\mathbb{S}^2)}

\newcommand{\g}{\mathfrak{g}} %Lie algebra

\newcommand{\algebra}{\mathfrak{su}(N)}

\newcommand{\normLtwoN}[1]{\norm{#1}_{L^2_N}}
\newcommand{\normLtwo}[1]{\ensuremath{\norm{#1}_{L^2}}}
\newcommand{\normLoneN}[1]{\norm{#1}_{L^1_N}}

\newcommand{\normLinfN}[1]{\norm{#1}_{L^\infty_N}}
\newcommand{\normLinf}[1]{\norm{#1}_{L^\infty}}
\newcommand{\normHmoneN}[1]{\norm{#1}_{H^{-1}_N}}

\newcommand{\normHoneN}[1]{\norm{#1}_{H^{1}_N}}
\newcommand{\normHone}[1]{\norm{#1}_{H^{1}}}
\newcommand{\normHtwo}[1]{\norm{#1}_{H^{2}}}
\newcommand{\normHfive}[1]{\norm{#1}_{H^{5}}}

\begin{document}

\title[Eulerian and Lagrangian stability in Zeitlin's model]{Eulerian and Lagrangian stability in Zeitlin's model of hydrodynamics}

\author{Klas Modin${^{1\armaonly{,*}}}$}
\address{${^1}$Department of Mathematical Sciences, Chalmers University of Technology and the University of Gothenburg, SE-412~96 Gothenburg, Sweden}
\arxivonly{
\email{klas.modin@chalmers.se}
}
\armaonly{
\thanks{${}^*$Corresponding author, \texttt{klas.modin@chalmers.se}, \href{https://orcid.org/0000-0001-6900-1122}{ORCID:0000-0001-6900-1122}}
}
% \thanks{The first author was supported by the Swedish Research Council (grant number 2022-03453) and the Knut and Alice Wallenberg Foundation (grant number WAF2019.0201).}

\author{Manolis Perrot${^{2\armaonly{,\dagger}}}$}
\address{${^2}$Univ. Grenoble Alpes, CNRS, Inria, Grenoble INP, LJK, 38000 Grenoble, France}
\arxivonly{
\email{manolis.perrot@univ-grenoble-alpes.fr}
}
\armaonly{\thanks{${}^\dagger$\texttt{manolis.perrot@univ-grenoble-alpes.fr}, \href{https://orcid.org/0009-0005-7297-3703}{ORCID:0009-0005-7297-3703}}}
% \thanks{The second author was supported by a PhD fellowship from Ecole Normale Supérieure Paris.}

%    General info

\date{\today}

% \dedicatory{This paper is dedicated to our advisors.}


\begin{abstract}
    The two-dimensional (2-D) Euler equations of a perfect fluid possess a beautiful geometric description: they are reduced geodesic equations on the infinite-dimensional Lie group of symplectomorphims with respect to a right-invariant Riemannian metric. This structure enables insights to Eulerian and Lagrangian stability via sectional curvature and Jacobi equations.

    The Zeitlin model is a finite-dimensional analog of the 2-D Euler equations; the only known discretization that preserves the rich geometric structure. Theoretical and numerical studies indicate that Zeitlin's model provides consistent long-time behaviour on large scales, but to which extent it truly reflects the Euler equations is mainly open. Towards progress, we give here two results.
    First, convergence of the sectional curvature in the Euler--Zeitlin equations on the Lie algebra $\mathfrak{su}(N)$ to that of the Euler equations on the sphere.
    Second, $L^2$-convergence of the corresponding Jacobi equations for Lagrangian and Eulerian stability.
    The results allow geometric conclusions about Zeitlin's model to be transferred to Euler's equations and vice versa, which might be central in the ultimate aim: to characterize the generic long-time behaviour in perfect 2-D fluids.
% This paper is a sample prepared to illustrate the use of the American
% Mathematical Society's \LaTeX{} document class \texttt{amsart} and
% publication-specific variants of that class for AMS-\LaTeX{} version 2.
\end{abstract}

\maketitle

%%%%%%%%%%%%%%%%%%%%%%%%%%%%%%%%%%%%%%%%%%%%%%%%%%%%%%%%%%%%%%%%%%%%%%%%
% \footnote{Here is an example of a footnote. Notice that this footnote
% text is running on so that it can stand as an example of how a footnote
% with separate paragraphs should be written.
% \par
% And here is the beginning of the second paragraph.}%
%%%%%%%%%%%%%%%%%%%%%%%%%%%%%%%%%%%%%%%%%%%%%%%%%%%%%%%%%%%%%%%%%%%%%%%%

\arxivonly{
\tableofcontents
\newpage
}

\armaonly{
\subsection*{Statements and Declarations}
The authors have no relevant financial or non-financial interests to disclose.
Data sharing is not applicable to this article as no datasets were generated or analysed during the current study.

% \medskip

\subsection*{Acknowledgements}
The first author was supported by the Swedish Research Council (grant number 2022-03453) and the Knut and Alice Wallenberg Foundation (grant number WAF2019.0201).
The second author was supported by a PhD fellowship from Ecole Normale Supérieure Paris.
The authors would also like to thank Petra Flurin for her help on quantization estimates.

\newpage
}

\section{Introduction}

The motion of an incompressible fluid can be described in two equivalent ways: either by following the trajectory of a fixed fluid particle---the Lagrangian point of view---or by considering the velocity of fluid particles passing by a fixed point in space---the Eulerian point of view. 
The two viewpoints give rise to two notions of stability with respect to perturbation of the initial conditions.

In his pioneering work, Arnold~\cite{arnold_sur_1966} showed that ideal fluid motions describe geodesics on the group of volume-preserving diffeomorphisms endowed with a right-invariant $L^2$ Riemannian metric. 
This beautiful description enables tools from Riemannian geometry in the study of ideal fluids.
Now, geodesic motion on a Riemannian manifold is stable in the Lagrangian sense if a perturbation of the initial conditions yields a geodesic that stays close to the unperturbed one. 
Infinitesimal perturbations of geodesics are given by the Jacobi fields, whose evolution is governed by the sectional curvature. 
Roughly speaking, negative curvature suggests instability, whereas positive curvature suggests stability. 
This lead to an interest in sectional curvature of volume-preserving diffeomorphism groups on various domains. 
Using Fourier series, Arnold first derived a formula for the sectional curvature when the fluid domain is the two-torus. 
Lukatsky \cite{lukatskii_curvature_1979}, Arakelyan-Savvidy
\cite{arakelyan_geometry_1989}, Dowker-Wei \cite{dowker_area-preserving_1990} and Yoshida \cite{yoshida_riemannian_1997} then derived corresponding formulae for the sphere, Nakamura \cite{nakamura_geodesics_1992} for the three-torus case, whereas Lukatsky \cite{lukatskii_structure_1988} and Preston \cite{preston_nonpositive_2005} exposed general computations for two-dimensional (2-D) compact surfaces. 
Regarding the link with stability, Misio\l ek~\cite{misiolek_stability_1993} adapted Rauch comparison theorem to prove that negative sectional curvature of a plane spanned by the velocity and a Jacobi field implies that the $L^2$ norm of the Jacobi field grows at least linearly in time. 

Misio\l ek's result established ``slow'' (less than exponential) Lagrangian instability in time.
Arnold, however, early advocated that mostly negative sectional curvature should imply ``fast'' (exponential) instability and based thereon he concluded that long-term weather forecasts are intrinsically unreliable (see \cite[Preface and Ch.~4B]{arnold_topological_1999}).
The question was further clarified by Preston \cite{preston_eulerian_2002}, who demonstrated that, although the sectional curvature is non-positive and mostly strictly negative, the Jacobi fields do not necessarily grow ``fast''. 
More precisely, he studied a splitting of the Jacobi equations and proved that (i) the sign of sectional curvature alone could not provide information about exponential Lagrangian instability, and (ii) exponential Eulerian instability always imply Lagrangian exponential instability. 
In summary, the notion of Lagrangian and Eulerian stability and their connection to sectional curvature is an important tool in the analysis of ideal hydrodynamics.

Another important tool for understanding 2-D ideal hydrodynamics is Zeitlin's model~\cite{zeitlin_finite-mode_1991, zeitlin_self-consistent_2004}, which in turn is based on quantization results of Hoppe \cite{hoppe_quantum_1982, Ho1989}.
Zeitlin's model is the only known (spatial) finite-dimensional approximation that fully adopts Arnold's geometric description in that it also describes geodesics on a Lie group equipped with a right-invariant Riemannian metric.
This structure gives rise to conservation of Casimirs, such as \emph{enstrophy}, critical for the understanding of 2-D-specific turbulence phenomenology as described by Kraichnan~\cite{kraichnan_inertial_1967}.
% Higher order Casimir also play a role in the formation of large-scale coherent vortex structures \cite{abramov_statistically_2003}.
In this context, Zeitlin's model provide a coherent approach to simulating the qualitative long-time behavior of 2-D Euler equations \cite{modin_casimir_2020, CiViMo2023}.
Furthermore, since Zeitlin's model establish a link between hydrodynamics and matrix theory, results from the latter enable new techniques.
An example is canonical decomposition of the vorticity field along the stabilizer of the stream function, which remarkably capture the dynamics of large and small vortex scales formations~\cite{modin_canonical_2021}.
Numerical evidence show that Zeitlin's model, contrary to traditional disretizations, retain the correct qualitative behavior, for example the spectral power laws in the inverse energy cascade~\cite{CiViLuMoGe2022}. 
Local convergence of solutions to Zeitlin's model to solutions of the Euler equations was established by Gallagher~\cite{Ga2002}.
But a rigorous understanding of the observed superior long-time behavior remains largely open.

Our objective here is to answer the following question: how well does Zeitlin's model capture the Lagrangian and Eulerian stability properties of the 2-D Euler equations?
In particular, does the sectional curvature in Zeitlin's discretization converge to the sectional curvature of the Euler equations?
We answer this and related questions for the case when the fluid domain is the sphere; the most relevant domain for applications in geophysical fluid dynamics (\emph{cf.}~\cite{pedlosky_geophysical_2013, zeitlin_geophysical_2018}).

% Ideal hydrodynamics on the two-sphere, although being an idealized model, encompass many features of large-scale geophysical flows \cite{pedlosky_geophysical_2013, zeitlin_geophysical_2018}. It is thus a fundamental model in geophysical fluid dynamics. Its specific Lie-Poisson structure leads to an infinity of constant of motions, called Casimirs. The quadratic Casimir is called \textit{enstrophy}, and is of paramount importance to understand the 2D turbulence phenomenology \cite{kraichnan_inertial_1967}. Higher-order Casimirs are also though to play a role in the formation of large-scale coherent vortexes \cite{abramov_statistically_2003}. Zeitlin proposed a spatial discretization of Euler equation on the sphere \cite{zeitlin_self-consistent_2004} based on a previous work of Hoppe in the context of relativistic membranes \cite{hoppe_quantum_1982}. This semidiscrete system has the remarkable property of being a finite-dimensional Lie-Poisson system, thus ensuring the exact conservation of energy and Casimirs. 

In \autoref{sec: euler geometry} we recall the geometric formulation of Euler equations on the sphere. 
In \autoref{sec: qtzed euler} we then present Zeitlin's finite-dimensional analog of the continuous Euler equations. 
In \autoref{sec: stability results} we recall the notions of Eulerian and Lagrangian stability and the link to sectional curvature. 
As a first main result, we show in \autoref{sec: proof securv} that the sectional curvature of the Euler--Zeitlin equations converges to the sectional curvature of the Euler equations when the degrees of freedom in the model tend to infinity. 
In particular, this result implies that for Zeitlin's model with enough degrees of freedom, the sign of the finite- and infinite-dimensional sectional curvatures are the same. 
Thus Zeitlin's model preserve the Lagrangian stability behavior implied by the sectional curvature. 
The second main result, proved in \autoref{sec: proof split eq thm}, concerns stationary solutions of both the Euler--Zeitlin and the Euler equations. 
We show that Lagrangian perturbations (\textit{i.e.}\ Jacobi fields) and Eulerian perturbations of the Zeitlin system converge in a certain sense toward corresponding perturbations of the continuous Euler equations. 
Thus, Zeitlin's model also preserves the stable or unstable nature of stationary solutions to the Euler equations.
We are not aware of any other discretization of the 2-D Euler equations with this property.

\arxivonly{
\subsection*{Acknowledgements}
The first author was supported by the Swedish Research Council (grant number 2022-03453) and the Knut and Alice Wallenberg Foundation (grant number WAF2019.0201).
The second author was supported by a PhD fellowship from Ecole Normale Supérieure Paris.
The authors would also like to thank Petra Flurin for her help on quantization estimates.
}

\section{Euler's equations on the sphere}\label{sec: euler geometry}

In his seminal work, Arnold~\cite{arnold_sur_1966} described ideal incompressible fluid flows on a Riemannian manifold $(M,g)$ with volume form $\mu$, as an equation for geodesics on the infinite-dimensional group of volume preserving diffeomorphisms $\operatorname{Diff}_{\mu}(M)$, with respect to the standard $L^2$ metric. 
More precisely, Arnold showed that the geodesic equation on $\operatorname{Diff}_{\mu}(M)$, right translated to the Lie algebra $\mathfrak{X}_\mu(M)$ of divergence-free vector fields, yields the incompressible Euler equations
\begin{equation}\label{eq: euler}
\begin{aligned}
    &\dot u + \nabla_u u = - \nabla p\\
    &\operatorname{div}u =0 ,
\end{aligned}
\end{equation}
where $\dot u \coloneqq \partial_t u$ denotes differentiation with respect to time, and $\nabla_u u$ denotes the co-variant derivative of $u$ along $u$.
The equivalent Hamiltonian formulation on the cotangent bundle is described by the \textit{kinetic energy} Hamiltonian, which is also right-invariant. 
Hamilton's equations can then be reduced by translation to the corresponding Lie-Poisson system on the smooth dual of the Lie algebra $\mathfrak{X}_\mu(M)^* \simeq \Omega^1(M)/\Omega^0(M)$ (\textit{cf.}~\cite{arnold_topological_1999}).

When $M$ is the unit sphere $\mathbb{S}^2$, there is a Lie--Poisson isomorphism between $\mathfrak{X}_\mu(M)^*$ and the space $C^\infty_0(\mathbb{S}^2)$ of smooth function with vanishing mean. 
The isomorphism is given by $\star d$, where $\star$ is the Hodge star and $d$ is the exterior derivative. 
When identifying a one-form to a divergence-free velocity field via the metric, the aforementioned isomorphism consists in taking the $\operatorname{curl}$ of the velocity field which gives the \textit{vorticity function} $\omega \in C^\infty_0(\mathbb{S}^2)$. 
The significance of the vorticity function is that it is transported by the vector field $u$.
In turn, the divergence free vector field $u$ is the Hamiltonian vector field $X_\psi$ for some Hamiltonian $\psi \in C_0^\infty(\mathbb{S}^2)$ called the \textit{stream function}.
The Euler equations \eqref{eq: euler} can then be written entirely in terms of the vorticity and stream function
\begin{align}\label{eq: vorticity eq}
    \dot{\omega} = \pbra{\psi,\omega}, \quad \Delta\psi = \omega .
\end{align}
The Poisson bracket on $\mathbb{S}^2$ is given by
\begin{align}
    \pbra{\psi,\omega}(x) = x \cdot (\nabla \psi \times \nabla \omega )
\end{align}
where the gradients are taken in $\mathbb{R}^3$, extending constantly on rays functions on the sphere. 
The vorticity equation \eqref{eq: vorticity eq} is an infinite-dimensional Lie--Poisson system on $C^\infty_0(\mathbb{S}^2)$ for the Lie-Poisson bracket given by 
\begin{align}
    \pbra{\mathcal{F},\mathcal{G}}_{LP}(\omega) = \int_{\mathbb{S}^2} \omega \pbra{\frac{\delta \mathcal{F}}{\delta \omega},\frac{\delta \mathcal{G}}{\delta \omega}} 
\end{align}
where $\mathcal{F},\mathcal{G}$ are functionals $\colon C^\infty_0 (\mathbb{S}^2) \rightarrow \mathbb{R}$, and for the specific Hamiltonian 
\begin{align}\label{eq: hamiltonian_omega}
    \mathcal{H}(\omega) = - \frac{1}{2} \int_{\mathbb{S}^2} \omega \Delta^{-1} \omega .
\end{align}
% which is a constant of motion. 
In addition to total energy, and contrary to the 3-D Euler equations, the 2-D system possesses infinitely many constants of motion, called \textit{Casimirs}: for any function $f\in C^\infty (\mathbb{R},\mathbb{R})$,
\begin{align}
    \mathcal{C}_f(\omega) = \int_{\mathbb{S}^2} f(\omega)
\end{align}
is conserved.
A distinguished class of Casimirs are those with $f(x) = x^k \; (k\in \mathbb{N})$. 
In particular, the quadratic Casimir is called \textit{enstrophy}, and, as mentioned, is critical in 2-D turbulence phenomenology \cite{kraichnan_inertial_1967}. 
Higher-order Casimirs are also though to play a role in the formation of large-scale coherent vortex structures \cite{abramov_statistically_2003}. 

% To see it explicitly, given $u$ a solution of \eqref{eq: euler}, the corresponding vorticity obeying \eqref{eq: vorticity eq} is obtained by computing $\omega = (\textrm{curl } u )\cdot n$, where $n$ is the outward pointing normal unit vector. Conversely, let $\omega$ be a solution of \eqref{eq: vorticity eq}, then the associated divergence-free velocity field solution of Euler equation is retrieved by first computing the \textit{stream function} $\psi := \nabla^{-1} \omega$, then computing the corresponding Hamiltonian vector field $u = - X_{\psi}$ via the underlying simplectic structure. It is given explicitly on the sphere via the symplectic (rotated) gradient $u = - \Delta^\perp \psi$. It allows to understand the vorticity equation \eqref{eq: vorticity eq} as an advection equation of $\omega$.

Notice that the Hamiltonian $\mathcal{H}$ in equation \eqref{eq: hamiltonian_omega} is equal to the kinetic energy ($L^2$ norm) of the velocity field which in turn is the $H^1$ norm of the stream function
\begin{align}
    \mathcal{H}(\omega) = -\frac{1}{2} \int_{\mathbb{S}^2} \omega \Delta^{-1} \omega = \frac{1}{2} \int_{\mathbb{S}^2} u \cdot u = -\frac{1}{2}\int_{\mathbb{S}^2}\psi\Delta\psi ,
\end{align}
where $u = X_\psi$.
The connection between solutions to the vorticity equation \eqref{eq: vorticity eq} and geodesics on $\operatorname{Diff}_\mu(\mathbb{S}^2)$ is established as follows: if $\omega(t)$ is a solution and $\psi(t)$ the corresponding path of stream functions, then a geodesic curve $\gamma(t) \in \operatorname{Diff}_\mu(\mathbb{S}^2)$ is obtained by integrating the corresponding non-autonomous ordinary differential equation
\begin{equation}\label{eq: reconstruction}
    \dot\gamma(t) = X_{\psi(t)}\circ\gamma(t) .
\end{equation}

% and the associated metric on $\mathrm{diff}_{\mathrm{div}}(\mathbb{S}^2)$ is inducing a $H^{-1}$ metric on the dual Lie algebra $C^\infty_0(\mathbb{S}^2)$
% \begin{align}
%     \an{u,u'}_{\mathrm{diff}_{\mathrm{div}}(\mathbb{S}^2)} = \int u\cdot u' = \int (-\Delta^{-1} \omega) \omega' = \an{\omega,\omega'}_{H^{-1}} 
% \end{align}
% Thus the Lie-Poisson system \eqref{eq: vorticity eq} is still describing geodesics. Consequently we would further investigate sectional curvature directly on the Lie lagebra, which greatly simplifies computations.

\section{Zeitlin's model on the sphere}\label{sec: qtzed euler}
% Quantization theory aims at replacing the Poisson algebra of functions on a symplectic manifold with a sequence of operator algebras in such a way that the commutator of operators approximates the Poisson bracket.
% For compact domains, the operators can be matrices.
Zeitlin's insight was to use quantization theory to spatially discretize the vorticity equation \eqref{eq: vorticity eq} by replacing the Poisson algebra of smooth functions with the matrix Lie algebra $\mathfrak{u}(N)$ of skew-Hermitian $N\times N$ complex matrices \cite{zeitlin_finite-mode_1991,zeitlin_self-consistent_2004}.
To achieve this, Zeitlin used an explicit quantization scheme developed by Hoppe \cite{hoppe_quantum_1982,Ho1989} initially within the context of relativistic membranes.
% introduced finite-dimensional Lie-Poisson system approximating vorticity equation \cite{zeitlin_finite-mode_1991,zeitlin_self-consistent_2004}, based on an approximation of the algebra of smooth function by finite-dimensional matrix algebras, introduced by Hoppe \cite{hoppe_quantum_1982} in the context of relativistic membranes. 
Hoppe's quantization is an example of \textit{Toeplitz quantization} \cite{bordemann_gl_1991, bordemann_toeplitz_1994}. 
% Recently, Modin and Viviani \cite{modin_casimir_2020} used it to perform more reliable long-time simulations of vorticity equations. 

% \subsection{Discretization of the Euler equations} \label{subsec: dicretization} 

\subsection{$L_\alpha$ approximation}
 Bordemann, Hoppe, Schaller, and Schlichenmaier~\cite{bordemann_gl_1991} proposed a set of axioms to characterize a family of matrix algebras $(\mathfrak{g}_N,\sbra{\cdot,\cdot}_N)$ as an approximation of an arbitrary Lie algebra $(\g,[\cdot,\cdot])$ (typically an infinite-dimensional Poisson algebra).
 They called it \textit{$L_\alpha$ approximation}. 
 Given, for each Lie algebra $\g_N$, a distance $d_N$ and a projection $p_N\colon \g_N \longrightarrow \g$, the family $(\g_N,[\cdot,\cdot]_N,d_N,p_N)$ is an $L_\alpha$ approximation of $\g$ if for each pair $x,y\in\mathfrak{g}$
\begin{enumerate}
    \item $d_N(p_N(x),p_N(y)) \to 0$ as $N\rightarrow \infty$ implies $x=y$, and
    \item $d_N([p_N(x),p_N(y)]_N - p_N([x,y])) \rightarrow 0$ as $N\rightarrow \infty$ .
\end{enumerate}

% {\color{red}KM: Do we really need to talk about $L_\alpha$ convergence? I don't think we ever use it.}

\subsection{Quantization of the sphere}
Recall the $L^2$ orthonormal basis for $C^\infty(\mathbb{S}^2,\C)$ provided by the \textit{spherical harmonics}.
Expressed in inclination-azimuthal coordinates $(\theta,\phi)\in [0,\pi]\times [0, 2\pi)$ they are
\begin{align}\label{eq: spherical harmonics}
    \mathcal{Y}_{lm}(\theta,\phi) = \sqrt{\frac{2l+1}{4\pi}\frac{(l-m)!}{(l+m)!}}P_{lm}\big(\operatorname{cos}(\theta)\big)e^{im\phi}
\end{align}
where $P_{lm}$ are the associated Legendre polynomials for $l\geq 0$ and $m\in\{-l,\cdots,l\}$. 
The spherical harmonics are eigenfunctions of the Laplace--Beltrami operator on $\mathbb{S}^2$
\begin{align}
    \Delta \mathcal{Y}_{lm} = -l(l+1) \mathcal{Y}_{lm}.
\end{align}
Using the spherical harmonics basis, Hoppe~\cite{hoppe_quantum_1982} gave in his thesis an explicit quantization of the Poisson algebra $\mathfrak{g} = (C^\infty(\mathbb{S}^2,\C),\{\cdot,\cdot\} )$ of complex valued smooth functions (see \cite[Example 3]{bordemann_gl_1991} for an exposition in terms of $L_\alpha$ approximations). 
The approximating Lie algebras are $\mathfrak{g}_N = \mathfrak{gl}(N,\C)$ with $[\cdot,\cdot]_N =  \frac{1}{\hbar_N}[\cdot,\cdot]$, where $$\hbar_N = \frac{2}{N-1}$$ and $[\cdot,\cdot]$ is the matrix commutator. 

Lets introduce the following rescaled inner products on $\mathfrak{gl}(N,\C)$
\begin{align}
    & \an{A, B}_{L^2_N} \coloneqq \frac{4\pi}{N}\operatorname{tr} (A^\dagger B).
\end{align}
The distances $d_N$ are given by the induced norms of the inner products, and the projections are defined via 
$p_N\mathcal{Y}_{lm} :=iT^N_{lm}\in\mathfrak{gl}(N,\C)$ for
\begin{align}\label{eq: quantized hamronic}
    (T^N_{lm})_{m1,m2} = \sqrt{\frac{N}{4\pi}}(-1)^{(N-1)/2-m}\sqrt{2l+1}
\begin{pmatrix}(N-1)/2 & l & (N-1)/2 \\
-m_1 & m & m_2\end{pmatrix}
\end{align}
where $(:::)$ is the Wigner 3j-symbol. Note that the quantized harmonics $(T^N_{lm})$ form an orthonormal basis of $\mathfrak{gl}(N,\C)$ with respect to the \textit{rescaled} inner product. 
Thus, a function expanded in spherical harmonics $f = \sum_{l=0}^{\infty} \sum_m f^{lm}\mathcal{Y}_{lm}$ is projected onto
\begin{equation}
    p_N(f) = \sum_{l=0}^{N-1}\sum_{m=-l}^l f^{lm}\, iT^N_{lm} .
\end{equation}

Recall from \autoref{sec: euler geometry} that the vorticity and stream functions for the Euler equations on the sphere are real valued.
That is, we consider the Poisson sub-algebra $C^\infty(\mathbb{S}^2)$ of real valued functions.
The corresponding Lie sub-algebras of $\mathfrak{gl}(N)$ are given by skew-Hermitian complex matrices $\mathfrak{u}(N)$.
Furthermore, the smaller, trace free sub-algebras $\mathfrak{su}(N)\subset \mathfrak{u}(N)$ correspond to $C^\infty_0(\mathbb{S}^2)$.

% In fact if we will only focus on the quantization of real zero-mean functions, which are mapped via the projections $p_N$ onto $\mathfrak{su}(N)$, \textit{i.e.} anti-hermitian matrices with zero trace.
% \begin{remark}
%     For readers familiar with quantum mechanics, it is classical to "map" the Poisson bracket $\pbra{\cdot,\cdot}$ to the rescaled bracket $-i/\hbar \sbra{\cdot,\cdot}$, where here $1/\hbar = N^{3/2}/\sqrt{16 \pi}$ and the $i$ factor have been included in the projection.
% \end{remark}

The last ingredient we need in order to approximate the vorticity formulation \eqref{eq: vorticity eq} of the 2-D Euler equations is an approximation $\Delta_N \colon \mathfrak{su}(N)\to \mathfrak{su}(N)$ of the Laplace--Beltrami operator.
Given the construction above, it is natural to define it directly in terms of the basis $T_{lm}^N$ as
\begin{align}\label{eigen}
    \Delta_N T^N_{lm} = -l(l+1)T^N_{lm}
\end{align}
so that, up to truncation $l < N$, it corresponds to the Laplacian on $C^\infty(\mathbb{S}^2)$.
It turns out that the quantized Laplacian $\Delta_N$ defined this way admits a beautiful, canonical description in the theory of unitary representation theory of $\mathfrak{so}(3)$; see the work of Hoppe and Yau~\cite{hoppe_properties_1998} for details.
In short, if $x^1,x^2,x^3\in C^\infty(\mathbb{S}^2)$ denotes the Euclidian coordinate functions (for $\mathbb{S}^2$ embedded as the unit sphere in $\mathbb{R}^3$), and if $X^i_N = p_N x^i$, then $X^1_N,X^2_N,X^3_N$ are (scaled) generators of a representation of $\mathfrak{so}(3)$ in $\mathfrak{u}(N)$ and the quantized Laplacian for $F\in\mathfrak{u}(N)$ is given by%
\arxivonly{\footnote{The following calculation ensures that the sign is correct:
\begin{align*}
    \operatorname{tr}([X,F]^\dagger [X,F]) &= -\operatorname{tr}([X,F][X,F])
    = -\operatorname{tr}([X,F](XF-FX)) \\
    &= -\operatorname{tr}(([X,F]X-X[X,F]) F) \\
    &= -\operatorname{tr}(-[X,[X,F]] F) \\
    &= \operatorname{tr}(\underbrace{-[X,[X,F]}_{\sim \; -\Delta_N}]^\dagger F).
\end{align*}
}}
\begin{equation}\label{eq: quantized laplacian formula}
    \Delta_N F = \frac{1}{\hbar_N^2}\sum_{i=1}^3 [X_N^i,[X_N^i, F]] .
\end{equation}

% {\color{red}KM: check sign:
% \begin{align*}
%     \operatorname{tr}([X,F]^\dagger [X,F]) &= -\operatorname{tr}([X,F][X,F])
%     = -\operatorname{tr}([X,F](XF-FX)) \\
%     &= -\operatorname{tr}(([X,F]X-X[X,F]) F) \\
%     &= -\operatorname{tr}(-[X,[X,F]] F) \\
%     &= \operatorname{tr}(\underbrace{-[X,[X,F]}_{\sim \; -\Delta_N}]^\dagger F) \\
% \end{align*}
% }

We now have all the components we need to define Zeitlin's model on the sphere.
For more details on the projection operator $p_N$ and the Hoppe--Yau Laplacian $\Delta_N$, including efficient computer implementation, we refer to the work of Modin and Viviani~\cite{modin_casimir_2020} and Cifani, Modin, and Viviani~\cite{CiViMo2023}.

\subsection{The Euler--Zeitlin equations}

Using the aforementioned matrix algebra approximation of $\big( C^\infty(\mathbb{S}^2),\{\cdot,\cdot\}\big)$, the \textit{Euler--Zeitlin equations} are given by the following matrix flow for $W = W(t) \in \mathfrak{su}(N)$
\begin{align}\label{eq: qtzed vorticity equation}
    \dot W = \frac{1}{\hbar_N}\sbra{P,W},\quad \Delta_N P = W ,
\end{align}
where $W=W(t)\in\mathfrak{su}(N)$ is the \textit{vorticity matrix} and $P = P(t)\in \mathfrak{su}(N)$ is the \textit{stream matrix}.
The remarkable feature of these equations is that they completely capture Arnold's description, but in a finite-dimensional setting: the Euler--Zeitlin equations constitute a Lie--Poisson system on the dual $\mathfrak{su}(N)^*$, which via the Frobenius inner product is identified with $\mathfrak{su}(N)$, for the Hamiltonian given by
\begin{align}
    \mathcal H_N(W) = \frac{1}{2}\an{W,-\Delta_N^{-1}W}_{L^2_N} = - \frac{2 \pi}{N} \operatorname{tr}(W^\dagger \Delta^{-1}_N W).
\end{align}
% where $4 \pi / N$ is a scaling factor in order to ensure that $\mathcal{H}_N (p_N \omega) \underset{N \rightarrow \infty}{\longrightarrow} \mathcal{H}(\omega)$ (see section \ref{sec: estimates}).
% {\color{red}KM: Maybe use different scaling here.}
% \blu{MP: done}
The associated right invariant Riemannian metric on the matrix group is determined from the inner product on the Lie algebra $\mathfrak{su}(N)$ given by
\begin{align}
    \an{A,B}_{H^{-1}_N} \coloneqq  \an{A,-\Delta_N^{-1} B}_{L^2_N} = -\frac{4 \pi}{N}\operatorname{tr}( A^\dagger \Delta_N B).
\end{align}
The reconstruction equation for geodesics on the matrix group $\operatorname{SU}(N)$ is
\begin{equation}
    \dot G(t) = P(t)G(t) .
\end{equation}
This equation is the direct analog to equation \eqref{eq: reconstruction} for geodesics on the infinite-dimensional group $\operatorname{Diff}_\mu(\mathbb{S}^2)$.

The Euler--Zeitlin system \eqref{eq: qtzed vorticity equation} is isospectral, which reflects the transport nature of the vorticity equation \eqref{eq: vorticity eq}.
Indeed, in addition to the Hamiltonian, the system has the following constants of motion (Casimirs)

\begin{align}
    \mathcal{C}_{N,k}(W) = \frac{4\pi}{N}\operatorname{tr}(W^k).
\end{align}
% {\color{red}KM: I think the right formula is 
% \begin{align}
%     \mathcal{C}_{N,k}(W) = \frac{4\pi}{N}\operatorname{tr}(W^k).
% \end{align}
% Proof: $W=I$ corresponds to $\omega = 1$ and $\mathcal{C}_k(1) = \int_{\mathbb{S}^2} 1^k = 4\pi$.
% }\blu{OK}
As $N\to \infty$ these converge to the corresponding Casimirs $\mathcal C_k(\omega)$ of the continuous system (see \cite[Corollary 8.1.2]{rios_symbol_2014}).

\section{Stability results}\label{sec: stability results}

\subsection{Lagrangian stability and sectional curvature}
For a fluid domain given by a compact Riemannian manifold $M$, let us first review the notion of Lagrangian stability on $\operatorname{Diff}_\mu(M)$ and how it is related to sectional curvature.
Consider a geodesic $\gamma = \gamma(t)$ on $\operatorname{Diff}_{\mu}(M)$, and a smooth family $\gamma_s$ of geodesics such that $\gamma_0 = \gamma$.
The corresponding \emph{Jacobi field} is the vector field along $\gamma$ defined by 
\begin{equation}\label{eq: jacobi field}
J(\gamma(0)) = \frac{d}{ds} \Big|_{s=0} \gamma_s  \in T_\gamma \operatorname{Diff}_{\mu}(M) .
\end{equation}
It satisfies the \emph{Jacobi equation}, which in abstract notation can be written
\begin{align}\label{eq: jacobi abstract}
    \bar\nabla_{\dot{\gamma}}\bar\nabla_{\dot{\gamma}} J + R_\gamma(J,\dot{\gamma})\dot{\gamma} = 0,
\end{align}
where $\bar\nabla_{\dot\gamma}$ is the ``big'' co-variant derivative on $\operatorname{Diff}_\mu(\mathbb{S}^2)$ (as an infinite-dimensional Riemannian manifold) and $R$ is the corresponding ``big'' Riemann curvature tensor
% , defined as $R(u,v)w = \nabla_u \nabla_v w - \nabla_v \nabla_u w - \nabla_{\sbra{u,v}} w$ 
(see \cite{misiolek_stability_1993} for details on the functional analytic setting).
The fluid motion $\gamma$ is \emph{Lagrangian stable} if every Jacobi field along $\gamma$ remains bounded in the Riemannian metric along the trajectory $t\mapsto \gamma(t)$. 

From the curvature tensor one can extract \emph{sectional curvature} for the plane spanned by two tangent vectors $U,V\in T_\varphi\operatorname{Diff}_\mu(M)$ as 
\begin{align}
    C_\varphi(U,V) \coloneqq \frac{\an{R_\varphi(U,V)V , U}_{\varphi}}{\lVert U \rVert_{\varphi}^2\lVert V \rVert_{\varphi}^2 - \langle U,V\rangle_{\varphi}^2} ,
\end{align}
where $\an{\cdot,\cdot}_\varphi$ denotes the Riemannian metric and $\lVert\cdot\rVert_{\varphi}$ the corresponding norm.
Due to right invariance, it follows that $C_\varphi(U,V) = C_{\operatorname{id}}(u,v) \eqqcolon C(u,v)$ for the divergence free vector fields on $M$ given by $u=U\circ\varphi^{-1}$ and $v = V\circ\varphi^{-1}$.
Furthermore,
% in terms of $\langle \cdot,\cdot\rangle \coloneqq \langle \cdot,\cdot\rangle_{\operatorname{id}}$ and $\lVert \cdot \rVert \coloneqq \langle \cdot,\cdot\rangle$, 
Arnold gave the following explicit formula 
\begin{equation}
\begin{aligned}
    C(u,v) &= \frac{1}{4}\an{ P\nabla_u v + P\nabla_v u  }_{L^2}^2 + \frac{1}{2}\an{ [u,v], \nabla_u v - \nabla_v u }_{L^2} \\
    & - \frac{3}{4}\an{[u,v] }_{L^2}^2 - \langle P\nabla_u u, \nabla_v v\rangle_{L^2},
\end{aligned}
\end{equation}
where $P$ denotes $L^2$ orthogonal projection onto the divergence free part (the \emph{Leray projection}).
For the case $M=\mathbb{S}^2$, with $u=X_f$ and $v=X_g$ for $f,g\in C^\infty_0(\mathbb{S}^2)$, the formula for sectional curvature is 
\begin{equation}\label{eq: sec curv S2}
    \begin{aligned}
        C(X_f,X_g) &= \frac{1}{4}\| \Delta^{-1}\{\Delta f, g\} + \Delta^{-1}\{\Delta g, f\}\|_{H^1}^2 \\
        & + \frac{1}{2}\an{\{f,g \}, \Delta^{-1}\{\Delta f, g\}-\Delta^{-1}\{\Delta g, f\}}_{H^1} \\
        & - \frac{3}{4}\| \{f,g \} \|_{H^1}^2 
        - \an{\Delta^{-1}\{\Delta f, f\},\Delta^{-1}\{\Delta g, g\}}_{H^1} ,
    \end{aligned}
\end{equation}
where $\langle \cdot,\cdot\rangle_{H^{1}} = \langle -\Delta(\cdot),\cdot \rangle_{L^2}$ and correspondingly for $\lVert \cdot\rVert_{H^{1}}$.
A direct analog of this formula in the finite-dimensional case yields the sectional curvature $C_N\colon \mathfrak{su}(N)\times\mathfrak{su}(N)\to \mathbb{R}$ of $\operatorname{SU}(N)$ in Zeitlin's model (see \autoref{sec: proof securv} below).

Sectional curvature provide important information about the Lagrangian stability of the fluid.
For example, Misiolek~\cite[Lemma~4.2]{misiolek_stability_1993} proved that if the sectional curvature of the plane spanned by $\dot{\gamma}$ and a Jacobi field $J$ remains \textit{nonpositive} along $\gamma$, then $J$ grows at least linearly in time, and, consequently, $\gamma$ is at least weakly Lagrangian unstable. 
Similar results are valid also in finite dimensions, as pointed out by Preston~\cite{preston_eulerian_2002}.

% However, this notion of linear instability is weak compared to exponential growth usually referred to when dealing with instabilities. 

We are now ready to state the first main theorem of the paper.
\begin{theorem}\label{thm: seccurv}
    The sectional curvature of $\operatorname{SU}(N)$ (with Zeitlin's metric) converges to the sectional curvature of $\operatorname{Diff}_\mu(\mathbb{S}^2)$ (with Arnold's metric) as follows: for any $f,g \in H^7(\mathbb{S}^2)$
    \begin{equation*}
        \abs{C_N(p_Nf, p_N g) - C(X_f, X_g)} \leq \hbar_N c_0\norm{f}_{H^7}^2\norm{g}_{H^7}^2 
        % \quad\text{as}\quad N\to \infty 
        ,
    \end{equation*} 
    where the constant $c_0>0$ is independent of $f,g,N$.
%    , the sectional curvature of the plane spanned by $p_N f$ and $p_N g$ in $(\mathfrak{su}(N), , \sbra{\cdot,\cdot}_N,\an{\cdot,\cdot}_{H^{-1}(N)})$ converges toward the sectional curvature of the plane spanned by $f$ and $g$ in $(C^\infty_0(\mathbb{S}^2), \pbra{\cdot,\cdot},\an{\cdot,\cdot}_{H^{-1}})$ when $N \rightarrow \infty$, :
%    \begin{align}
%        C(p_N f, p_N g)  \underset{N \rightarrow \infty}{\longrightarrow} C(f,g)
%    \end{align}
\end{theorem}

\subsection{Splitting of Jacobi equation and Eulerian stability}
Just as the Euler equations are expressed in the Eulerian (right reduced) variable $u=\dot\varphi\circ\varphi^{-1}$, it is natural to express the corresponding Jacobi equation~\eqref{eq: jacobi abstract} also in Eulerian variables.
As explored by Preston~\cite{preston_eulerian_2002,Pr2004}, such a rewriting splits the second order Jacobi equation into two first order equations that exposes the links between Eulerian and Lagrangian stability.
% Indeed, Preston proposed to split the second order Jacobi equation into two first order equations. 
For the family of geodesics $\gamma_s$ as before, let $u_s = \dot{\gamma_s}\circ \gamma^{-1}_s$ be the associated Eulerian velocity fields, each one a solution to the Euler equations. 
% With $J$ the Jacobi field as in \eqref{eq: jacobi field}, 
The right reduced version of the Jacobi field $J$ in equation~\eqref{eq: jacobi field} is $j = J \circ \gamma^{-1}$. 
Since the Jacobi equation is of second order, we also need a variable corresponding to the derivative of $J$, namely the variation of the velocity field 
\begin{equation}
    z = \frac{d}{ds}\Big|_{s=0} u_s .
\end{equation}
Under this change of variables $(J,\dot J)\leftrightarrow (j, z)$, the Jacobi equation~\eqref{eq: jacobi abstract} turns into the first order system 
\begin{subequations}\label{eq: reduced jacobi general}
\begin{align}
    & \dot j + \sbra{u,j} = z \label{eq: jacobi reconstruction}
    \\
    & \dot z + P \nabla_u z + P \nabla_z u = 0 \label{eq: lin euler}
\end{align}
\end{subequations}
As expected, equation \eqref{eq: lin euler} is the linearized Euler equation about the solution $u$, whereas equation \eqref{eq: jacobi reconstruction} is a reconstruction equation for the reduced Jacobi field $j$.

On the sphere we obtain a ``vorticity formulation'' expressed in the functions $\upsilon,\zeta \in C^\infty_0(\mathbb{S}^2)$ defined by $j = X_\upsilon$ and $\zeta = \operatorname{curl} z$.
Geometrically, $\upsilon$ is an element of the Lie algebra, whereas $\zeta$ is an element of the dual.
% Expressed in these variables the system \eqref{eq: reduced jacobi general} 
The equations \eqref{eq: reduced jacobi general} expressed in these variables become\arxivonly{\footnote{The equations follow from \eqref{eq: reduced jacobi general} via the identities: (i) $[X_f,X_g] = -X_{\{f,g \}}$, (ii) $\operatorname{curl} X_{\Delta^{-1}\zeta} = \zeta =  \operatorname{curl}z$, and (iii) $\dot X_{\upsilon} + [X_\psi,X_{\upsilon}] = X_{\Delta^{-1}\zeta}$ $\iff$ 
$\dot\upsilon - \{\psi, \upsilon \} = \Delta^{-1}\zeta$.
}}
\begin{equation}\label{eq: preston split omega}
\begin{aligned}
    &\partial_t \upsilon - \pbra{\Delta^{-1}\omega,\upsilon} = \Delta^{-1}\zeta
    \\
    &\partial_t \zeta - \pbra{\Delta^{-1}\omega,\zeta} - 
    \pbra{\Delta^{-1} \zeta,\omega} = 0,
\end{aligned}
\end{equation}
where $\omega$ fulfill the vorticity equation \eqref{eq: vorticity eq}.
Now, the solution $\omega=\omega(t)$ is called \textit{Eulerian stable} (with respect to some norm) if every solution $\zeta$ of the system~\eqref{eq: preston split omega} is bounded uniformly in time.
% (With respect to some norm, for example the $H^{-1}$ norm, which is equivalent to $z$ in equation \eqref{eq: lin euler} being bounded in the $L^2$ norm.)

\begin{remark}
    % Note that this splitting has been generalized in \cite{khesin_curvatures_2011} (section 4) 
    On an arbitrary Lie group $G$ endowed with a right-invariant metric, the corresponding splitting of the Jacobi equations become the following system of equations on the direct product of $\mathfrak{g}=T_e G$ with itself:
    \begin{align*}
        & \dot Y + [\mathcal{A}^{-1}W, Y] = \mathcal{A}^{-1}Z
        \\
        & \dot Z - \operatorname{ad}^*_{\mathcal{A}^{-1}W} Z - \operatorname{ad}^*_{\mathcal{A}^{-1}Z} W = 0,
    \end{align*}
    where $W=W(t)\in \mathfrak{g}$ is a solution to the Euler--Arnold equation for geodesic curves on $G$, $\operatorname{ad}^*_X$ is the adjoint of the linear operator $\operatorname{ad}_X = \sbra{X, \cdot}$, and $\mathcal A\colon \mathfrak{g}\to\mathfrak{g}^*$ is the \emph{inertia operator} that defines the inner product on $\mathfrak{g}$.
    % with respect to the inner product on $T_e G$, \emph{i.e.}, $\an{\operatorname{ad}^*_X Y ,Z} =\an{Y, [X,Z]}$.
    For details, see for example \cite[sec.~4]{khesin_curvatures_2011}.
    % {\color{red}KM: I think these equation can only be true when the metric i bi-invariant. One can see it because the metric tensor doesn't appear anywhere in the equations and the dual variable $Z$ is treated as an element in the Lie algebra. We need to check this. Maybe it is better to remove this remark.}
    % \\
    % \blu{MP: I feel like the explainations in the paper \cite{khesin_curvatures_2011} page 11 are quite convincing. $z \in \mathfrak{g} = T_e G$ because $z = \partial_s u^s \restrict{s=0}$ and $u^s = \dot{\gamma^s}(\gamma^s)^{-1}$. So the second equation is just a perturbation/linearization of Euler-Arnold equation on $T_eG$. Sign difference compared to the paper is due to the sign difference in definition of the $ad^*$ }
    % \\
    % {\color{red}KM: Ok, I see now. The previous $\operatorname{ad}^*$ is really $B$ since $\an{\cdot,\cdot}$ is the inner product, not pairing with dual. 
    % Then the formulae make sense. 
    % However, this definition of $\operatorname{ad}^*$ is non-standard, so I rewrote it instead in terms of the standard $\operatorname{ad}^*$ operator to avoid confusion. Notice that equations \eqref{eq: preston split omega} and \eqref{eq: qtzed split jacobi} use this convention.
    % }
\end{remark}

For the Euler--Zeitlin equations \eqref{eq: qtzed vorticity equation} on $\mathfrak{su}(N)$ we get the analog of \eqref{eq: preston split omega}
\begin{equation}\label{eq: qtzed split jacobi}
  \begin{aligned}
    &\dot{Y} - \frac{1}{\hbar_N}\sbra{\Delta_N^{-1}W,Y} = \Delta_N^{-1}Z
    \\
    &\dot{Z} - \frac{1}{\hbar_N}\sbra{ \Delta^{-1}_N W, Z } - \frac{1}{\hbar_N}\sbra{ \Delta^{-1}_N Z, W } = 0 .
\end{aligned}
\end{equation}
The interpretation of these equations is two-fold: they describe at the same time a discretization of Preston's reduced Jacobi equations \eqref{eq: preston split omega} and, independently of the connection to the Euler equations, the reduced form of the Jacobi equations for the Zeitlin model describing geodesics on $\operatorname{SU}(N)$.
% There are two equivalent interpretation of this system: it is either a "quantization" of the continuous split Jacobi equations \eqref{eq: preston split omega}, or the split Jacobi equations of the finite-dimensional system on $\mathfrak{su}(N)$.

To study the correspondence of Eulerian and Lagrangian stability between the Euler and Euler--Zeitlin systems we now restrict to \textit{stationary solutions} of both systems.
It means that
\begin{align}
    \pbra{\Delta^{-1} \omega,\omega} &= 0 \label{eq: stationary euler}\\
    \frac{1}{\hbar_N}[\Delta_N^{-1} p_N\omega,p_N \omega] &= 0. \label{eq: stationary euler zeitlin}
\end{align}
It is an open problem to characterize in which situations stationary Euler~\eqref{eq: stationary euler} implies stationary Euler--Zeitlin~\eqref{eq: stationary euler zeitlin}
(although we always have that if $\omega$ is stationary then equation \eqref{eq: stationary euler zeitlin} is fulfilled in the limit $N\to \infty$).
However, the following situations are straightforward to check (\emph{cf.}~Viviani~\cite{viviani_symplectic_2020}):

\begin{itemize}
    \item If, for some fixed $l$, the vorticity is $\omega = \sum_{m=-l}^l \omega^{lm}\mathcal{Y}_{lm}$ then $\omega$ and $p_N\omega$ are stationary solutions.
    \item If $\omega$ is \emph{zonal} (there exists a choice of north pole on the sphere for which $\omega$ is constant on fixed latitudes) then $\omega$ and $p_N\omega$ are stationary solutions.
\end{itemize}

Define the embedding $\iota_N\colon \mathfrak{u}(N)\to C^\infty(\mathbb{S}^2)$ by 
\begin{equation}\label{eq: iota embedding}
    \iota_N(T^N_{lm}) = \mathcal{Y}_{lm}, \quad l=0,\ldots,N-1 .
\end{equation}
Our second main theorem states that corresponding stationary solutions to the Euler and the Euler--Zeitlin equations share the same Eulerian and Lagrangian $L^2$-stability as $N\to\infty$.
\begin{theorem}\label{th: eulerian and lagrangian}
   Let $\omega$ be a stationary solution of the vorticity equation \eqref{eq: vorticity eq} such that $W = p_N \omega$ is a stationary solution of the Euler--Zeitlin equation \eqref{eq: qtzed vorticity equation}.
   Let $\upsilon(t)$ and $\zeta(t)$ be solutions of the reduced Jacobi equations \eqref{eq: preston split omega}.
   Furthermore, let $Y(t)$ and $Z(t)$ be corresponding solutions of the finite-dimensional reduced Jacobi equations \eqref{eq: qtzed split jacobi} with $Y(0) = p_N\upsilon(0)$ and $Z(0) = p_N \zeta(0)$.
%    Let $\omega^Y(t)$ and $\omega^Z(t)$ be solutions of the split Jacobi equation with initial data $\omega^Y_0$ and $\omega^Z_0$. Let also $Y_N(t)$ and $Z_N(t)$ be solutions of the quantized split Jacobi equations \eqref{eq: qtzed split jacobi} with compatible initial data (see \autoref{sec: proof split eq thm}). \\
   Then, for any fixed $t$,
   \begin{equation*}
    \begin{aligned}
        &\normLtwo{ \iota_N Y(t) - \upsilon(t)} \longrightarrow 0
        \\
        &\normLtwo{ \iota_N Z(t) - \zeta(t) } \longrightarrow 0
   \end{aligned}  
   \quad\text{as}\quad N\to \infty  .
   \end{equation*}
   Moreover, the convergence is uniform on bounded intervals of $t$.
\end{theorem}

One interpretation is that the Zeitlin discretization~\eqref{eq: qtzed split jacobi} of the reduced Jacobi equations~\eqref{eq: preston split omega} is convergent in the sense of numerical analysis.
Indeed, the proof is based on concepts of stability and consistency and is given in \autoref{sec: proof split eq thm}. 
The theorem implies that Euler--Zeitlin model preserves the stable or unstable nature of stationary solutions of Euler equations, so that results on either the continuous or the discretized system can be transferred to the other.
For example, the result by Taylor~\cite[Thm. 4.1.1]{taylor_euler_2016}, that zonal stationary solutions that are strictly monotonous in the meridional direction are Eulerian stable.
%  (\textit{e.g.} Taylor \cite[Thm. 4.1.1]{taylor_euler_2016}){\color{red}[find reference]}, a direct consequence is the following.

% \begin{corollary}
%     Let $\omega$ be a zonal stationary solution (for example, a rotationally invariant vortex blob, or two opposite vortex blobs). 
%     The corresponding stationary solution $W=p_N \omega$ of the Euler--Zeitlin equation \eqref{eq: qtzed vorticity equation} is stable in the Eulerian sense.
%     % {\color{red}KM: Perhaps this requires uniform in time estimates. Check.}
% \end{corollary}

% \begin{proof}
%     It is a consequence of the \emph{stability} (see \autoref{sec: proof split eq thm} below) of solutions to the discretized equations \eqref{eq: qtzed split jacobi}, namely
%     \begin{equation*}
%         \normLtwoN{Y(t)}  \leq \alpha \normLtwo{\upsilon(t)} \quad\text{and}\quad
%         \normLtwoN{Z(t)}  \leq \beta \normLtwo{\zeta(t)} ,
%     \end{equation*}
%     for constants $\alpha,\beta>0$ independent of $t$.
% \end{proof}

\section{Bracket convergence and preliminary estimates}\label{sec: estimates}
Before proving the two main theorems, we expose a central result used in the proofs, and then give some needed estimates.

The projection $p_N$ can be understood as Toeplitz quantization on the sphere. 
The central result of interest to us is that the Poisson bracket is approximated by the Lie algebra bracket in a stronger sense than the one shown by Hoppe.
To state it we first introduce the matrix operator norm, given for $A\in\mathfrak{u}(N)$ by
\begin{equation}
    \normLinfN{A}\coloneqq \underset{\norm{x}=1}{\text{sup}}\norm{Ax} 
\end{equation}
where $\norm{\cdot}$ denotes the Euclidian norm on $\C^N$.
The notation of the norm is motivated by the following consistency result.

\begin{theorem}[Bordemann, Meinrenken, Schlichenmaier~\cite{bordemann_toeplitz_1994}]\label{thm: bordemann norm estimates}
    For every $f\in C^\infty(\mathbb{S}^2)$ there exists $c>0$ such that
    \begin{equation*}
        \normLinf{ f } - c\hbar_N \leq \normLinfN{ p_N f} \leq \normLinf{ f } .
    \end{equation*}
\end{theorem}

In the same paper the authors also prove convergence of the brackets as $\hbar_N\to 0$.
Indeed, they prove the following result:

\begin{theorem}[Bordemann, Meinrenken, Schlichenmaier~\cite{bordemann_toeplitz_1994}]
\label{th: bracket_cv} 
% For $W\in \mathfrak{u}(N)$ let $\|W\|_{L^\infty,N}:=   \underset{\|x\|=1}{\text{sup}}\|Ax\|$ denote the operator norm, where $\|\cdot\|$ is the Hermitian norm on $\C^N$. 
% Then, 
For every $f,g \in C^\infty(\mathbb{S}^2)$
\begin{align*}
    \normLinfN{ \frac{1}{\hbar_N}[p_Nf,p_Ng] - p_N\pbra{f,g}} = O(\hbar_N).
\end{align*}
% \blu{Still not clear if it is the Kostant-Souriau (="geometric") or the Toeplitz quantization...}
% {\color{red}KM: it is Toeplitz. However, we can also obtain Kostant-Souriau this way by applying some scaled Laplacian (see the remark on page 5 in the 2001 paper by Schlichenmaier).}
\end{theorem}

Charles and Polterovich \cite[Prop.~3.6, 3.9]{charles_sharp_2018} exposed a sharper estimate of this quantum-classical correspondence, namely
\begin{multline}
    \normLinfN{\frac{1}{\hbar_N}[p_Nf,p_Ng] - p_N\pbra{f,g} }  \leq \\ \hbar_N c_0 (\norm{ f }_{C^1} \norm{ g }_{C^3} + \norm{ f }_{C^2} \norm{ g }_{C^2} + \norm{ f }_{C^3} \norm{ g }_{C^1})
\end{multline}
where $c_0>0$ is a constant independent of $f$ and $g$, and $\norm{f}_{C^k} = \underset{i \leq k}{\operatorname{max}} \; \operatorname{sup} \abs{\nabla^i f}$. 
% \blu{if $(k-r)/n > 1/q$, then $W^q_k \subset C^r$. For us, $k-r > n/q = 1$}.
Via Sobolev embeddings 
%\cite[Thm.~2.10, 2.20]{aubin_nonlinear_1998}, 
we then obtain the following bound
\begin{align}\label{eq: sharp bound}
    \normLinfN{\frac{1}{\hbar_N}[p_Nf,p_Ng] - p_N\pbra{f,g} }  \leq \hbar_N c_0 \normHfive{ f } \normHfive{ g } ,
\end{align}
where $\normHfive{\cdot}$ is a Sobolev $H^5$ norm (with a suitable, fixed scaling).

% \blu{MP: The main explaination is that the subprincipal symbol (term of order $N^{-3/2}$ in the serie contains only derivatives less or equal than two. If i understand well it is eq (6) of Charles \cite{charles_subprincipal_2016}, I have a call with my friend next week such that she can explained it better to me. In the case of Kostant Souriau quantization, you can even get an $H^1$ norm.}
% {\color{red}KM: $H^1$ would be interesting as it is the limiting case.}

In addition to the spectral norm $\normLinfN{\cdot}$, we shall use the following 
% We now give lemmata concerning the following 
inner products and norms on $\algebra$
\begin{align}
    \an{A, B}_{L^2_N} &\coloneqq \frac{4\pi}{N}\operatorname{tr} (A^\dagger B),
        &\normLtwoN{A} &\coloneqq \sqrt{\an{A, A}_{L^2_N}}
    \\
     &
        &\normLoneN{A} &\coloneqq \frac{4\pi}{N}\sum_{k=1}^N \abs{\lambda_k} 
    \\
    \an{A, B}_{L^2_N} &\coloneqq \frac{4\pi}{N}\operatorname{tr} (A^\dagger B),
        &\normLtwoN{A} &\coloneqq \sqrt{\an{A, A}_{L^2_N}}
    \\
    \an{A, B}_{H^1_N} &\coloneqq -\frac{4\pi}{N}\operatorname{tr} (A^\dagger \Delta_N B),
        &\normHoneN{A} &\coloneqq \sqrt{\an{A, A}_{H^1_N}}    
    \\
    \an{A, B}_{H^{-1}_N} &\coloneqq -\frac{4\pi}{N}\operatorname{tr} (A^\dagger \Delta^{-1}_N B),
        &\normHmoneN{A} &\coloneqq \sqrt{\an{A, A}_{H^{-1}_N}}
\end{align}
where $\lambda_1,\ldots,\lambda_N$ are the eigenvalues of $A\in\algebra$.
The first result is that the inner products converge spectrally to their continuous analogs.

% Then we show that they converge toward the continuous inner products, since the projection $p_N$ can be seen as a truncation in spherical harmonics decomposition which are eigenvectors of the Laplacian:

\begin{lemma}\label{prop: cv norms}
    Let $s>1$.
    For every $f,g \in H^s(\mathbb{S}^2)$
    \begin{align*}
        & \abs{\an{p_N f, p_N g}_{L^2_N} - \an{ f, g}_{L^2}} \leq 2\hbar_N^s\norm{f}_{H^s}\norm{g}_{H^s} \\
        & \abs{\an{p_N f, p_N g}_{H^{-1}_N} - \an{ f, g}_{H^{-1}}}  \leq 2\hbar_N^{s+1}\norm{f}_{H^{s}}\norm{g}_{H^{s}} \\
        & \abs{\an{p_N f, p_N g}_{H^1_N} - \an{ f, g}_{H^1}}  \leq 2\hbar_N^{s-1}\norm{f}_{H^{s}}\norm{g}_{H^{s}}
    \end{align*}
\end{lemma}

\begin{proof}
The embedding $\iota_N$, defined in equation \eqref{eq: iota embedding}, is isometric for any pair $H^s_N,H^s$ of inner products on $\mathfrak{u}(N)$ and $C^\infty(\mathbb{S}^2)$.
The estimates then follow from standard results for convergence of the spherical harmonics expansion, which we repeat for completeness:

Let $\Pi_N$ denote the $L^2$--projection on spherical harmonics with $l\geq N$. Then
\begin{align*}
    \abs{\an{\Pi_N f,g}_{L^2}} &= \abs{\sum_{l=N}^{\infty}\sum_{m=-l}^l f^{lm}g^{lm}} = \abs{\sum_{l=N}^{\infty}\frac{1}{(l(l+1))^s}\sum_{m=-l}^l (l(l+1))^s f^{lm}g^{lm}} \\
    &\leq  \norm{f}_{H^s}\norm{g}_{H^s} \left(\sum_{l={N}}^{\infty}\frac{1}{l(l+1)}\right)^s \\
    &= \norm{f}_{H^s}\norm{g}_{H^s} \left(\sum_{l={N}}^{\infty}\frac{1}{l}- \frac{1}{l+1}\right)^s \\ &= \norm{f}_{H^s}\norm{g}_{H^s} \frac{1}{N^s}\leq 2 \hbar_N^s \norm{f}_{H^s}\norm{g}_{H^s}.
\end{align*}
Analogous estimates for the $H^{-1}$ and $H^1$ norms conclude the proof.
\end{proof}

The second lemma concerns the comparison of the different inner products and norms.
\begin{lemma}\label{lem: norms}
    For every $A,B\in \mathfrak{su}(N)$
    \begin{align*}
      (i)&\quad \normHmoneN{ A } \leq \frac{1}{\sqrt 2}\normLtwoN{A}
      \\
      (ii)&\quad  \normLtwoN{\Delta_N^{-1} A}  \leq \frac{1}{2}\normLtwoN{A}
      \\
      (iii)&\quad \normLtwoN{\Delta_N A}  \leq N(N+1)\normLtwoN{A}
      \\
      (iv)& \quad \lVert A \rVert_{L^1_N} \leq \sqrt{4\pi}\lVert A\rVert_{L^2_N}\leq 4\pi\lVert A\rVert_{L^\infty_N}
      \\
      (v)& \quad \abs{\an{A,B}_{L^2_N}} \leq \normLinfN{A}\normLoneN{B}.
     \end{align*}
\end{lemma}
\begin{proof}[Proof of Lemma~\ref{lem: norms}]
    $(i)$ The quantized harmonics $iT_{lm} = p_N \mathcal{Y}_{lm}$ form an $\an{\cdot, \cdot}_{L^2_N}$-orthonormal basis for $\algebra$. 
    If we expand $A\in\algebra$ in that basis
    \begin{align*}
        A = \sum_{l=1}^{N-1}\sum_{m=-l}^l a^{lm}\, iT_{lm}
    \end{align*}    
    we have
    \begin{align*}
     \normHmoneN{A}^2 &= \an{-\Delta_N^{-1}A,A }_{L^2_N} = \sum_{l=1}^{N-1}\sum_{m=-l}^l  \frac{(a^{lm})^2}{l(l+1)} 
     \\ & \leq \frac{1}{2}\sum_{l=1}^{N-1}\sum_{m=-l}^l (a^{lm})^2 = \frac{1}{2}\normLtwoN{A}^2.
    \end{align*}
    % For the second inequality, the operator norm is greater than the spectral radius (largest eigenvalue in module). 
    % Thus, if $i\lambda_1,\ldots,i\lambda_N$ denotes the eigenvalues of $A$, we have
    % \begin{align*}
    %     \frac{N}{4\pi}\normLtwoN{A}^2 = \sum_{k=1}^N \lambda_k^2 \leq N \normLinfN{A}^2 .
    % \end{align*}
    
    The estimates $(ii)$--$(iii)$ follow since $2 \leq l(l+1) \leq N(N+1)$.

    Let $i\lambda_1,\ldots,i\lambda_N$ denote the eigenvalues of $A$.
    The estimate $(iv)$ then follows since 
    \begin{align*}
        \normLoneN{A} &= \frac{4\pi}{N}\sum_{k=1}^N \abs{\lambda_k}
        \leq 4\pi \left( \frac{1}{N} \sum_{k=1}^N \abs{\lambda_k}^2 \right)^{1/2}
        \\ &= 
        \sqrt{4\pi} \normLtwoN{A} \leq 
        4\pi \normLinfN{A}.
    \end{align*}
    The last estimate is direct from Hölder's inequality for Schatten norms.
\end{proof}

% We can now restate the bracket estimate \eqref{eq: brack op estimate} as follows.

% \begin{corollary} \label{cor: brack op estimate}
%     There exists $\alpha>0$ such that for every $f,g\in C^\infty(\mathbb{S}^2)$
%     \begin{align*}
%         \normLtwoN{\frac{1}{\hbar_N}[p_Nf,p_Ng] - p_N\pbra{f,g} } \leq \alpha \hbar_N  \| f \|_{H^5} \| g \|_{H^5} .
%     \end{align*}        
% \end{corollary}

\section{Proof of the convergence of sectional curvature}\label{sec: proof securv}
In this section we prove \autoref{thm: seccurv}, which states that sectional curvature of the quantized matrix algebras $(\mathfrak{su}(N)$, $ \sbra{\cdot,\cdot}_N$, $\an{\cdot,\cdot}_{H^{1}_N})$ converges as $N\to\infty$ to the sectional curvature of infinite-dimensional system\arxivonly{\footnote{In the special case of functions expressed as finite combination of spherical harmonics, a more direct proof is given in \autoref{sec: finite combination} below.}}
\begin{equation*}
    (\mathfrak{X}_{\mu} (\mathbb{S}^2), \sbra{\cdot,\cdot},   \an{\cdot,\cdot}_{L^2}) \simeq (C^\infty_0(\mathbb{S}^2), \pbra{\cdot,\cdot}, \an{\cdot,\cdot}_{H^{1}}).  
\end{equation*}

Arnold \cite{arnold_sur_1966} gave the general formula for sectional curvature on a Lie group $G$ equipped with a right invariant Riemannian metric determined by an arbitrary inner product $\an{\cdot,\cdot}$ and its corresponding norm $\|\cdot\|$ on the Lie algebra $\mathfrak{g}$:

\begin{equation} \label{eq: Arnold Curvature}
    \begin{aligned}
        C(\xi,\eta) &= \frac{1}{4}\| B(\xi,\eta) + B(\eta,\xi)\|^2 + \frac{1}{2} \an{[\xi,\eta],B(\xi,\eta)-B(\eta,\xi)}\\ & - \frac{3}{4}\| [\xi,\eta]\|^2 - \an{B(\xi,\xi),B(\eta,\eta)},        
    \end{aligned}
\end{equation}
where $B\colon \mathfrak{g}\times\mathfrak{g}\to \mathfrak{g}$ is defined by
\begin{equation}
    \an{B(\xi,\eta), \nu} = \an{\xi,[\eta,\nu]}, \quad \forall\, \xi,\eta,\nu\in\mathfrak{g}.
\end{equation}

In our specific case, when $\mathfrak{g}$ is the Poisson algebra of smooth functions on the sphere and the metric is Sobolev $H^{1}$, we have 
\begin{align*}
    \an{B(f,g), h}_{H^1} &= \an{f,\{g,h\}}_{H^1} = \an{-\Delta f,\{g,h\}}_{L^2} \\ &= \an{\{ -\Delta f, g \}, h}_{L^2} = \an{\underbrace{\Delta^{-1}\{\Delta f, g\}}_{B(f,g)}, h}_{H^1}.
\end{align*}
Substitution into Arnold's formula~\eqref{eq: Arnold Curvature} yields equation \eqref{eq: sec curv S2} above for the sectional curvature $C(X_f,X_g)$.
It is convenient to express it in the $L^2$ inner product (because $L^2$ is bi-invariant with respect to the Poisson algebra)
\begin{subequations}\label{eq: sectional curvature Cinf}
\begin{align}
    & C(X_f,X_g) = \nonumber \\ &-\frac{1}{4}\an{ \{\Delta f, g\} + \{\Delta g, f\},\Delta^{-1}\Big(\{\Delta f, g\} + \{\Delta g, f\}\Big)}_{L^2} \label{eq: first term Cinf}\\
    & - \frac{1}{2}\an{\{f,g \}, \{\Delta f, g\}-\{\Delta g, f\}}_{L^2} \label{eq: second term Cinf} \\
    & + \frac{3}{4}\an{ \{f,g \},\Delta\{f,g \} }_{L^2} \label{eq: third term Cinf} \\
    & + \an{\{\Delta f, f\},\Delta^{-1}\{\Delta g, g\}}_{L^2} .  \label{eq: fourth term Cinf}   
\end{align}
\end{subequations}
The same calculation for $\mathfrak{su}(N)$ equipped with the right invariant metric determined by the inner product $\an{\cdot,\cdot}_{H^1_N}$ yields the analogous formula
\begin{subequations}\label{eq: sectional curvature suN}
    \begin{align}
        &C_N(F,G) = \nonumber \\  &-\frac{1}{4\hbar_N^2}\an{[\Delta_N F, G] +[\Delta_N G, F],\Delta_N^{-1}\Big([\Delta_N F, G] +[\Delta_N G, F]\Big)}_{L^2_N} \label{eq: first term} \\
        & - \frac{1}{2\hbar_N^2}\an{[F,G],[\Delta_N F, G]-[\Delta_N G, F]}_{L^2_N} \label{eq: second term} \\
        & + \frac{3}{4\hbar_N^2}\an{[F,G ],\Delta_N[F,G] }_{L^2_N} \label{eq: third term} \\
        & + \frac{1}{\hbar_N^2}\an{[\Delta_N F, F],\Delta_N^{-1}[\Delta_N G, G]}_{L^2_N} .  \label{eq: fourth term}  
    \end{align}    
\end{subequations}

Now, the aim is to prove that 
\begin{equation*}
    \abs{C(X_f,X_g)-C_N(p_N f, p_N g)}\to 0 \qquad\text{as}\qquad N\to \infty.
\end{equation*}
To this end, we carry out estimates for each corresponding term in the formulae \eqref{eq: sectional curvature Cinf} and \eqref{eq: sectional curvature suN}.
The ``problematic'' term is the third, since the $H^1_N$ norm is not bounded by $L^\infty_N$.

\subsection{The first, second, and fourth terms}
By construction of the quantized Laplacian we have that $p_N \circ \Delta = \Delta_N \circ p_N$.
Thus,
\begin{equation*}
    [\Delta_N p_N f, p_N g] = [p_N \Delta f, p_N g] \quad\text{and}\quad 
    [\Delta_N p_N g, p_N f] = [p_N \Delta g, p_N f].
\end{equation*}
If $\Pi_N\colon C^\infty(\mathbb{S}^2)\to C^\infty(\mathbb{S}^2)$ denotes the projection onto spherical harmonics with $l\geq N$, then for $a,b,c,d\in C^\infty_0(\mathbb{S}^2)$ and $q\in \{ 0,1\}$
\begin{align*}
    % 4 \abs{ &\text{\eqref{eq: first term}} -  \text{\eqref{eq: first term Cinf}}} = \\
    % 
    &
    \absbig{
        \anbig{\underbrace{\frac{1}{\hbar_N}[p_N a,p_N b]}_{\mathrm{I}},\underbrace{\frac{1}{\hbar_N}[p_N c,p_N d]}_{\mathrm{II}}}_{H^{-q}_N} - \anbig{\{\underbrace{a,b}_{\mathrm{I}'}\},\{\underbrace{c,d}_{\mathrm{II}'}\}}_{H^{-q}}
    } = 
    \\ &
    \absbig{
        \an{\mathrm{I},\mathrm{II}}_{H^{-q}_N} - \an{p_N \mathrm{I}',p_N \mathrm{II}'}_{H^{-q}_N} - \underbrace{\an{\Pi_{N}\{a,b\},\{c,d\}}_{H^{-q}}}_{r_N}
    } \leq
    \\ &
    \frac{1}{2}\absbig{
        \an{\mathrm{I}-p_N \mathrm{I}',\mathrm{II}+ p_N \mathrm{II}'}_{H^{-q}_N}
    }
    + 
    \frac{1}{2}\abs{ 
        \an{\mathrm{I}+p_N \mathrm{I}',\mathrm{II}- p_N \mathrm{II}'}_{H^{-q}_N} 
    } + \abs{r_N} \underset{\text{lem. \ref{lem: norms}}}{\leq}
    \\ &
    \frac{1}{2}
        \normLinfN{\mathrm{I}-p_N \mathrm{I}'}\normLoneN{\Delta_N^{-q}(\mathrm{II}+ p_N \mathrm{II}')}
    + 
    \\ &
    \qquad\qquad\frac{1}{2}
        \normLinfN{\mathrm{II}-p_N \mathrm{II}'}\normLoneN{\Delta_N^{-q}(\mathrm{I}+ p_N \mathrm{I}')}
    + \abs{r_N} \underset{\text{eq.~\eqref{eq: sharp bound}}}{\leq}
    % \\ &
\end{align*}
\begin{align*}
    &\frac{\hbar_N c_0}{2}
        % \normLinfN{\mathrm{I}-p_N \mathrm{I}'}
        \normHfive{a}\normHfive{b}
        \normLoneN{\Delta_N^{-q}(\mathrm{II}+ p_N \mathrm{II}')}
    + 
    \\ & \qquad
    \frac{\hbar_N c_0}{2}
        % \normLinfN{\mathrm{I}-p_N \mathrm{I}'}
        \normHfive{c}\normHfive{d}
        \normLoneN{\Delta_N^{-q}(\mathrm{I}+ p_N \mathrm{I}')}
    + \abs{r_N} \leq %\underset{\text{eq. \eqref{eq: sharp bound}}}{\leq}
    \\ &
    \frac{\hbar_N c_0}{2}
        % \normLinfN{\mathrm{I}-p_N \mathrm{I}'}
        \normHfive{a}\normHfive{b}
        \left(\normLoneN{\Delta_N^{-q}(\mathrm{II}- p_N \mathrm{II}')}+2\normLoneN{\Delta_N^{-q}p_N \mathrm{II}'}\right)
    + 
    \\ & \qquad
    \frac{\hbar_N c_0}{2}
        % \normLinfN{\mathrm{I}-p_N \mathrm{I}'}
        \normHfive{c}\normHfive{d}\left(\normLoneN{\Delta_N^{-q}(\mathrm{I}- p_N \mathrm{I}')}+
        2\normLoneN{\Delta_N^{-q}p_N \mathrm{I}'}\right)
    + \abs{r_N} \underset{\text{lem. \ref{lem: norms}}}{\leq}
\end{align*}
    
\begin{align*}
    % &
    % \hbar_N c_0\sqrt{\pi}
    %     % \normLinfN{\mathrm{I}-p_N \mathrm{I}'}
    %     \normHfive{a}\normHfive{b}
    %     \left(\normLtwoN{\Delta_N^{-q}(\mathrm{II}- p_N \mathrm{II}')}+2\normLtwoN{p_N\Delta^{-q} \mathrm{II}'}\right)
    % + 
    % \\ & \qquad
    % \hbar_N c_0\sqrt{\pi}
    %     % \normLinfN{\mathrm{I}-p_N \mathrm{I}'}
    %     \normHfive{c}\normHfive{d}\left(\normLtwoN{\Delta_N^{-q}(\mathrm{I}- p_N \mathrm{I}')}+
    %     2\normLtwoN{p_N \Delta^{-q} \mathrm{I}'}\right)
    % + \abs{r_N} \underset{\text{lem. \ref{lem: norms}}}{\leq}
    % \\ 
    &
    \frac{2\hbar_N c_0\sqrt{\pi}}{2^q}\Bigg(
        \normHfive{a}\normHfive{b}
        \left(\hbar_N c_0 \sqrt{\pi} \normHfive{c}\normHfive{d}+\normLtwoN{p_N \mathrm{II}'}\right)
    + 
    \\ & \qquad
    % \frac{2\hbar_N c_0\sqrt{\pi}}{2^q}
        \normHfive{c}\normHfive{d}\left(\hbar_N c_0 \sqrt{\pi} \normHfive{a}\normHfive{b}+
        \normLtwoN{p_N \mathrm{I}'}\right)\Bigg)
    + \abs{r_N} \underset{\text{thm. \ref{thm: bordemann norm estimates}}}{\leq}
    \\ &
    \frac{2\hbar_N c_0\sqrt{\pi}}{2^q}\Bigg(
        \normHfive{a}\normHfive{b}
        \left(\hbar_N c_0 \sqrt{\pi} \normHfive{c}\normHfive{d}+\normLtwo{\{c,d\}}\right)
    + 
    \\ & \qquad
    % \frac{2\hbar_N c_0\sqrt{\pi}}{2^q}
        \normHfive{c}\normHfive{d}\left(\hbar_N c_0 \sqrt{\pi} \normHfive{a}\normHfive{b}+
        \normLtwo{\{a,b\}}\right)\Bigg)
    + \abs{r_N} \underset{\text{lem. \ref{lem: classical comm est}}}{\leq}
    \\ &
    \frac{2\hbar_N c_0\sqrt{\pi}}{2^q}\Bigg(
        \normHfive{a}\normHfive{b}
        \left(\hbar_N c_0 \sqrt{\pi} \normHfive{c}\normHfive{d}+\normHone{c}\normHone{d}\right)
    + 
    \\ & \qquad
    % \frac{2\hbar_N c_0\sqrt{\pi}}{2^q}
        \normHfive{c}\normHfive{d}\left(\hbar_N c_0 \sqrt{\pi} \normHfive{a}\normHfive{b}+
        \normHone{a}\normHone{b}\right)\Bigg)
    + \abs{r_N}. %\underset{\text{lem. \ref{lem: classical comm est}}}{\leq}
\end{align*}
The last inequality follow from the well-known result:

\begin{lemma}\label{lem: classical comm est}
    Let $f,g \in C^1(\mathbb{S}^2)$. Then
    \begin{equation*}
        \normLtwo{\{f,g \}} \leq \normHone{f}\normHone{g} .
    \end{equation*}
\end{lemma}

\begin{proof}
    Let $J\colon T\mathbb{S}^2\to T\mathbb{S}^2$ denote the complex structure on $\mathbb{S}^2$.
    By Cauchy--Schwartz
    \begin{align*}
        \normLtwo{\{f,g \}}^2 = \int_{\mathbb{S}^2} \abs{\nabla f\cdot  J \nabla g}^2
        \leq \normLtwo{\nabla f}^2\normLtwo{J\nabla g}^2. 
    \end{align*}
    The estimate now follows since $\mathbb{S}^2$ is Kähler, so $J$ is an isometry.
\end{proof}

By choosing $a,b,c,d$ from the set $\{f,g,\Delta f, \Delta g \}$ to match the corresponding terms we obtain 
\begin{equation}\label{eq: most terms final estimate}
    \begin{aligned}
        &\abs{\text{\eqref{eq: first term}}-\text{\eqref{eq: first term Cinf}} + \text{\eqref{eq: second term}}-\text{\eqref{eq: second term Cinf}} + \text{\eqref{eq: fourth term}}-\text{\eqref{eq: fourth term Cinf}}}  \leq  \\
        &\qquad \alpha \hbar^2_N \norm{f}_{H^5}\norm{f}_{H^7} \norm{g}_{H^5}\norm{g}_{H^7} + \\ 
        &\qquad \beta \hbar_N \left(\sum_{a,b \in \{f,g\}}\norm{a}_{H^5}\norm{b}_{H^7} \right) \left(\sum_{a,b \in \{f,g\}}\norm{a}_{H^1}\norm{b}_{H^3} \right)  + \\
        &\qquad \gamma \sum_{a,b \in \{f,g\}} \normLtwo{\Pi_N \{ a,\Delta b\}}(\normLtwo{\Pi_N \{ b,\Delta a\}} + \normLtwo{\Pi_N \{ b, a\}} )
    \end{aligned}        
\end{equation}
for constants $\alpha>0$, $\beta>0$, and $\gamma>0$ independent of $f,g,N$.

% Notice that the formula~\eqref{eq: sec curv S2} above is the special case when $\mathfrak{g}$ is the Poisson algebra of smooth functions on the sphere and the metric is Sobolev $H^{1}$.

\subsection{The third term} 
The term 
\begin{equation}
    \abs{\text{\eqref{eq: third term}}-\text{\eqref{eq: third term Cinf}}} = \frac{3}{4}\abs{ \frac{1}{\hbar_N^2}\normHoneN{[p_N f,p_N g]}^2-\normHone{\{f,g\}}^2}
\end{equation}
cannot be treated directly by the estimate above, since the $H^1_N$ norm is not controlled by the $L^\infty_N$ norm as $N\to\infty$.
To overcome this problem we shall need the following results for the quantized Laplacian $\Delta_N$.

\begin{lemma}\label{lem: pN relatedness}
    For the coordinate functions $x^i\in C^\infty(\mathbb{S}^2)$ and the corresponding matrices $X^i_N = p_N x^i$, let $\nabla^{\bot,i} = \{x^i,\cdot \}$ and $\nabla_N^{\bot,i} = [X^i_N,\cdot]/\hbar_N$.
    Then $\nabla^{\bot,i}$ and $\nabla_N^{\bot,i}$ are $p_N$-related: 
    \begin{equation*}
        p_N \circ \nabla^{\bot,i} = \nabla^{\bot,i}_N \circ p_N.    
    \end{equation*}
\end{lemma}

\begin{proof}
    For any $l\geq 1$, the Lie group $\operatorname{SO}(3)$ acts on $\{ f\in C^\infty(\mathbb{S}^2)\mid \Delta f = l(l+1)f \}$ via its action on $\mathbb{S}^2$.
    The corresponding infinitesimal representation of $\mathfrak{so}(3)$ is Hamiltonian with respect to the Poisson bracket with generators  $x^1,x^2,x^3 \in C^\infty(\mathbb{S}^2)$, i.e., given exactly by $\nabla^{\bot,i}$.
    From the work of Hoppe and Yau~\cite{hoppe_properties_1998} it follows that $X_N^1,X_N^2,X_N^3$ are corresponding generators for the infinitesimal representation of $\mathfrak{so}(3)$ on $\{F\in \algebra \mid \Delta_N F = l(l+1)F \}$ and that $p_N$ is an algebra morphism.
    This proves the result.
\end{proof}

\begin{lemma}\label{lem: deltaN commutator relations}
    Let $F,G\in \algebra$ and $f,g\in C^3(\mathbb{S}^2)$. Then
    \begin{equation*}
        \Delta_N [F,G] = [\Delta_N F, G] + [F, \Delta_N G] + \sum_{i=1}^3 [\nabla^{\bot,i}_N F,\nabla^{\bot,i}_N G]
    \end{equation*}
    and
    \begin{equation*}
        \Delta \{f,g\} = \{\Delta f, g\} + \{f, \Delta g\} + \sum_{i=1}^3 \{ \nabla^{\bot,i} f,\nabla^{\bot,i} g\}.
    \end{equation*}
\end{lemma}

\begin{proof}
    The first result follows from the Hoppe--Yau formula~\eqref{eq: quantized laplacian formula}, application of the Jacobi identity twice, and the identity $\sum_{i=1}^3 (X_N^i )^2 = I$.
    The second result is a direct continous analog.
\end{proof}

From these results we obtain 
\begin{align*}
    &\abs{\text{\eqref{eq: third term}}-\text{\eqref{eq: third term Cinf}}} = \frac{3}{4}\abs{ \frac{1}{\hbar_N^2}\an{[p_N f,p_N g],\Delta_N [p_N f,p_N g]}_{L^2_N}^2-\an{\{f,g\},\Delta \{f,g\}}_{L^2}^2}
    \\
    & \underset{\text{lem. \ref{lem: deltaN commutator relations}}}{=} \frac{3}{4}\Bigg| \frac{1}{\hbar_N^2}\Big\langle[p_N f,p_N g], [\Delta_N p_N f,p_N g] + [p_N f,\Delta_N p_N g] \\ &\hspace{0.5\textwidth} + \sum_{i=1}^3 [\nabla^{\bot,i}_N p_N f,\nabla^{\bot,i}_N p_N g] \Big\rangle_{L^2_N}^2
    \\
    & \qquad - \Big\langle\{f, g\}, \{\Delta f, g\} + \{f,\Delta g\} + \sum_{i=1}^3 \{\nabla^{\bot,i} f,\nabla^{\bot,i} g\} \Big\rangle_{L^2}^2 \Bigg| 
    \\
    & \underset{\text{lem. \ref{lem: pN relatedness}}}{=} \frac{3}{4}\Bigg| \frac{1}{\hbar_N^2}\Big\langle[p_N f,p_N g], [p_N \Delta f,p_N g] + [p_N f,p_N \Delta g] \\ &\hspace{0.5\textwidth} + \sum_{i=1}^3 [p_N \nabla^{\bot,i} f,p_N \nabla^{\bot,i} g] \Big\rangle_{L^2_N}^2
    \\
    & \qquad - \Big\langle\{f, g\}, \{\Delta f, g\} + \{f,\Delta g\} + \sum_{i=1}^3 \{\nabla^{\bot,i} f,\nabla^{\bot,i} g\} \Big\rangle_{L^2}^2 \Bigg| .
\end{align*}
Each pair of corresponding quantized and continuous brackets can now be treated by the same estimate as the second term above, which gives
\begin{equation}\label{eq: third term final estimate}
\begin{aligned}
    &\abs{\text{\eqref{eq: third term}}-\text{\eqref{eq: third term Cinf}}}
    \leq  \\
    &\quad \alpha \hbar^2_N \norm{f}_{H^5}(\norm{f}_{H^7}+\norm{f}_{H^6}) \norm{g}_{H^5}(\norm{g}_{H^7}+\norm{g}_{H^6}) + \\ 
    &\quad \beta \hbar_N \left(\sum_{a,b \in \{f,g\}}\norm{a}_{H^5}(\norm{b}_{H^7}+\norm{b}_{H^6}) \right) \left(\sum_{a,b \in \{f,g\}}\norm{a}_{H^1}(\norm{b}_{H^3}+\norm{b}_{H^2}) \right)  + \\
    &\quad \gamma  \normLtwo{\Pi_N \{ f, g\}} \sum_{a,b \in \{f,g\}}\left(\normLtwo{\Pi_N \{ b,\Delta a\}} + \sum_{i=1}^3\normLtwo{\Pi_N \{\nabla^{\bot,i}a,\nabla^{\bot,i} b \}} \right) 
\end{aligned}
\end{equation}
for constants $\alpha>0$, $\beta>0$, and $\gamma>0$ independent of $f,g,N$.
From the Sobolev norm relations $\norm{f}_{H^q} \leq \norm{f}_{H^p}$ for $q\leq p$ and from Lemma~\ref{prop: cv norms} we then obtain the result in \autoref{thm: seccurv}.
This concludes the proof.

\section{Proof of the convergence of the reduced Jacobi equation} \label{sec: proof split eq thm}
In this section we prove the second main result stated in \autoref{th: eulerian and lagrangian} above.

Let us first rewrite the continuous and quantized reduced Jacobi equations \eqref{eq: preston split omega} and \eqref{eq: qtzed split jacobi} for corresponding stationary solutions $\omega_0 \in C_0^\infty(\mathbb{S}^2)$ and $W_0 = p_N \omega_0 \in \mathfrak{su}(N)$ as linear evolution equations on $C^\infty(\mathbb{S}^2 \times \mathbb{S}^2)$ and $(\mathfrak{su}(N))^2$ respectively. 
The continuous reduced Jacobi equation is then 
\begin{align}\label{eq: continuous EDO}
    \begin{cases}
    \dot{\bm{\xi}} = \mathcal{L} \bm{\xi}
    \\
    \bm{\xi}(t=0) = \bm{\xi}_0
    \end{cases}
\end{align}
where 
\[\bm{\xi} = \begin{bmatrix} \upsilon \\ \zeta \end{bmatrix} \quad \textrm{and} \quad \mathcal{L} = \begin{bmatrix}
    \pbra{\Delta^{-1}\omega_0, ~\cdot~} & \Delta^{-1} \\
    0 & \pbra{\Delta^{-1} \omega_0, ~\cdot~} + \pbra{\omega_0, \Delta^{-1} ~\cdot~}
\end{bmatrix}. \]
The corresponding quantized reduced Jacobi equation is
\begin{align}\label{eq: qtzed EDO}
   \begin{cases}
   \dot{\bm X} = \Lambda_N \bm{X}
    \\
    \bm{X} (t=0) = \bm{X}_{0} := p_N \bm \xi_0
   \end{cases} 
\end{align}
where

\begin{multline*}
    \bm{X} = \begin{bmatrix}
        Y \\ 
        Z
    \end{bmatrix} \quad \textrm{and} \\ \Lambda_N = \begin{bmatrix}
        \frac{1}{\hbar_N}\sbra{\Delta_N^{-1}W_0 , ~\cdot~} & \Delta_N^{-1}
        \\
        0 & \frac{1}{\hbar_N}\sbra{\Delta_N^{-1} W_0, ~\cdot~} + \frac{1}{\hbar_N}\sbra{\Delta_N^{-1} ~\cdot~, W_0 }
    \end{bmatrix}.
\end{multline*}

Recall that the embedding $\iota_N : \algebra \rightarrow \CS$ maps quantized harmonics to continuous ones. Let $\pi_N : \CS \rightarrow \CS$ be the truncation of the (continuous) spherical harmonics expansion up to $l \leq N-1$, \textit{i.e.}, $\pi_N = \iota_N p_N$. In order to directly compare the quantized and continous Jacobi equations, we introduce the operator $\mathcal{L}_N \coloneqq \iota_N\circ \Lambda_N\circ p_N$ such that if $\bm X(t)$ is a solution of the matrix dynamical system \eqref{eq: qtzed EDO}, then $\bm \xi_N (t)= \iota_N \bm X(t) + \Pi_N \bm \xi_0 $ is the solution of the continuous system
\begin{align}\label{eq: qtzed EDO2}
   \begin{cases}
   \dot{\bm \xi}_N = \mathcal L_N \bm{\xi}_N
    \\
    \bm{\xi}_N (t=0) = \bm \xi_0
   \end{cases} 
\end{align}
% \blu{MP: added $\bm \xi_N (t)= \iota_N \bm X(t) + \Pi_N \bm \xi_0 $ in order to generate a \textit{true} semigroup, ie $\mathcal{T}_N(0) = Id$.}
In order to prove the convergence $\bm \xi_N (t)$ to $\bm \xi (t)$, we will use the following result of Trotter and Kato (see, \textit{e.g.}, Pazy~\cite[Thm. 3.4.2]{pazy_semigroups_1983}).

\begin{theorem}[Trotter, Kato]\label{thm: trotter-kato}
    For a Banach space $(X,\norm{\cdot})$, let $(L_N)_{N \geq 0} $ and $L$ be linear operators $D(L) \subset X \to X$. 
    Let $\norm{\cdot}$ denote also the operator norm and make the following assumptions.
    \begin{description}
        \item[Well-posedness] there exist scalars $M, \omega$ such that $L$ is the infinitesimal generator of a $C^0$ semigroup $T(t)$ satisfying $ \norm{T(t)} \leq M e^{\omega t}$.
        \item[Stability] for any $N \geq 0$, $L_N$ is the infinitesimal generator of a $C^0$ semigroup $T_N(t)$ satisfying $ \norm{T_N(t)} \leq M e^{\omega t}$ (for the same $M$ and $\omega$).
        \item[Consistency] for every $x\in X$ and $\lambda \in \C$ with $\operatorname{Re} \lambda > \omega$  
        \begin{equation*}
            \norm{(I - \lambda L_N)^{-1}x - (I - \lambda L)^{-1}x} \to 0 \quad\text{as}\quad N\to \infty.
        \end{equation*}
        % \blu{OLD:
        % \begin{equation*}
        %     \norm{(I - \lambda L_N)^{-1}x - (I - \lambda L)^{-1}x} \leq \alpha(\lambda, x, \omega, M) \epsilon_N \quad\text{as}\quad N\to \infty,
        % \end{equation*}
        % where $\alpha(\lambda, x, \omega, M)>0$ is a numerical constant depending continuously on $\lambda, x, \omega, M$.
        % }
        % $(I - \lambda L_N)^{-1}x \rightarrow (I - \lambda L)^{-1}x $ as $N \rightarrow \infty$.  
    \end{description}
    Then, for every $x \in X$ and $t \geq 0$, 
    \begin{equation*}
        \norm{T_N(t) x - T(t) x} \to 0 \quad\text{as}\quad N\to \infty.
    \end{equation*}
    % \blu{OLD:
    % \begin{equation*}
    %     \norm{T_N(t) x - T(t) x} \leq e^{2\omega t} (1 + t) \beta(x, \omega, M) \epsilon_N
    % \end{equation*}
    % % $T_N(t) x \rightarrow T(t) x$ as $N \rightarrow \infty$. 
    % where $\beta(x, \omega, M)$ is a numerical constant depending continuously on $x, \omega, M$.
    % }
    % {\color{red}KM: It doesn't seem to be as easy as this. For one thing, the dependence of $\alpha$ on $x$ cannot be continuous. I suggest we stick with the standard result.}
    Moreover, the convergence is uniform on bounded intervals of $t$. 
\end{theorem}

\begin{remark}
    The classical formulation of the Trotter \& Kato result does not include convergence rates. 
    In the special case of Zeitlin's model, the convergence rate is $O(\hbar_N)$. 
    In \autoref{sec: proof trotter-kato} we have included a proof of \autoref{thm: trotter-kato} with this convergence rate.
\end{remark}

\subsection{Well-posedness of the continuous system}
Here we prove that $\mathcal{L}$ is the infinitesimal generator of a $C^0$ semigroup with respect to the $L^2$ norm.

\begin{proposition}\label{prop: well-posedness}
    The operator $\mathcal{L}\colon C^\infty(\mathbb{S}^2 \times \mathbb{S}^2)\to L^2(\mathbb{S}^2 \times \mathbb{S}^2)$ is the generator of a $C^0$ semigroup $\mathcal T(t)$ with 
    \begin{equation*}
        \normLtwo{\mathcal T(t)} \leq \exp t\sqrt{\left( \frac{1}{4}+\frac{1}{2}\normHone{\omega_0}^2 \right)} .
    \end{equation*}
\end{proposition}

\begin{proof}
    We split $\mathcal L = \mathcal L_1 + \mathcal L_2$, where 
    \begin{equation*}
        \mathcal{L}_1 = \begin{bmatrix}
            \pbra{\Delta^{-1}\omega_0, ~\cdot~} & 0 \\
            0 & \pbra{\Delta^{-1} \omega_0, ~\cdot~}
        \end{bmatrix}
        \qquad\text{and}\qquad
        \mathcal{L}_2 = \begin{bmatrix}
            0 & \Delta^{-1} \\
            0 & \pbra{\omega_0, \Delta^{-1} ~\cdot~}
        \end{bmatrix}.
    \end{equation*}
    Now, $\mathcal L_1$ is the generator for the semigroup $\mathcal T_1(t)$ given explicitly by 
    \begin{equation*}
        \mathcal T_1(t)\bm{\xi} = \bm{\xi}\circ\eta_t^{-1}
    \end{equation*}
    where $\eta_t$ is the diffeomorphism on $\mathbb{S}^2$ generated by the finite-dimensional Hamiltonian vector field $X_{\Delta^{-1}\omega_0}$ (the solution exists for all times since $X_{\Delta^{-1}\omega_0}$ is a smooth vector field on a compact domain). 
    Notice that $\mathcal T_1(t)$ is bounded with operator norm $1$, since
    \begin{equation*}
        \normLtwo{\mathcal T_1(t)\bm{\xi}}^2 = \normLtwo{\upsilon\circ\eta_t^{-1}}^2 + \normLtwo{\zeta\circ\eta_t^{-1}}^2 = \normLtwo{\upsilon}^2 + \normLtwo{\zeta}^2  =\normLtwo{\bm{\xi}}^2.
    \end{equation*}
    For the second part, $\mathcal L_2$ is a bounded operator, since
    \begin{multline*}
        \normLtwo{\mathcal L_2 \bm{\xi}}^2 = \normLtwo{\Delta^{-1}\zeta}^2 + \normLtwo{\{\omega_0,\Delta^{-1}\zeta \}}^2 \\
        \leq \frac{1}{4}\normLtwo{\zeta}^2 + \normHone{\omega_0}^2 \normHone{\Delta^{-1}\zeta}^2 \leq \left(\frac{1}{4}+\frac{1}{2}\normHone{\omega_0}^2\right)\normLtwo{\zeta}^2.
    \end{multline*}
    Thus, $\mathcal L_2$ generates a $C^0$ semigroup $\mathcal T_2(t)$.
    By the Lie--Trotter formula we then obtain the $C^0$ semigroup $\mathcal T(t)$ via 
    \begin{equation*}
        \mathcal T(t) = \lim_{n\to\infty} \left(\mathcal T_1(t/n)\mathcal T_2(t/n) \right)^n ,
    \end{equation*}
    which is generated by $\mathcal L$.
    The operator norm estimate also follows from the Lie--Trotter formula. 
\end{proof}

% , it is sufficient to show that $\mathcal{A}$ is a \emph{bounded linear operator} on $(\CS)^2$ with respect to $\| \cdot \|_{L^2}$.
% To this end, let $f$ and $g$ be any function of $\CS$. Then
    % \begin{align*}
    %      \normLtwoN{ A_N \begin{bmatrix}
    %          p_N f
    %          \\
    %          p_N g
    %      \end{bmatrix} }^2 
    % &= \normLtwoN{ \sbra{\Delta_N{-1} p_N \omega_0 , p_N f}_N + \Delta_N^{-1} p_N g }^2 \\ & +\normLtwoN{ \sbra{\Delta_N^{-1} p_N \omega_0, p_N g}_N + \sbra{ \Delta_N^{-1} p_N g, p_N \omega_0} }^2
    % \\
    % &\leq C_2^2 \normLtwoN{p_N \omega_0}^2 \normLtwoN{ \begin{bmatrix}
    %     p_N f
    %     \\
    %     p_N g
    % \end{bmatrix} }^2
    % \\
    % & \leq C_2^2 \normLtwo{\omega_0}^2 \normLtwo{\begin{bmatrix}
    %     f
    %     \\
    %     g
    % \end{bmatrix}}^2
    % \end{align*}

\subsection{Stability of the semidiscrete method}

We now prove that the quantized system \eqref{eq: qtzed EDO2} is stable, which is the semidiscrete correspondence to Proposition~\ref{prop: well-posedness}.

\begin{proposition}\label{prop: stability}
    For any $N\geq 0$, the operator $\mathcal{L}_N\colon C^\infty(\mathbb{S}^2 \times \mathbb{S}^2)\to L^2(\mathbb{S}^2 \times \mathbb{S}^2)$ is the generator of a $C^0$ semigroup $\mathcal T_N(t)$ with 
    \begin{equation*}
        \normLtwo{\mathcal T_N(t)} \leq \exp t\sqrt{\left( \frac{1}{4}+\frac{1}{2}\normHone{\omega_0}^2 \right)} .
    \end{equation*}
\end{proposition}

\begin{proof}
    The proof is a direct analog of the proof of Proposition~\ref{prop: well-posedness}.
    For the operator splitting $\mathcal L_N = \mathcal L_{N,1} + \mathcal L_{N,2}$, the semigroup for the generator $\mathcal L_{N,1}$ is 
    \begin{equation*}
        \mathcal T_{N,1}(t)\bm\xi = 
        \iota_N \begin{bmatrix}
            E_t Y E_t^\dagger \\
            E_t Z E_t^\dagger 
        \end{bmatrix},
    \end{equation*}
    where $E_t = \exp(t \Delta_N^{-1}W_0/\hbar_N)$, $Y = p_N \upsilon$, and $Z = p_N \zeta$.
    We then have
    \begin{equation*}
        \normLtwo{\mathcal T_{N,1}(t)\bm\xi}^2 = \normLtwoN{E_t Y E_t^\dagger}^2 + \normLtwoN{E_t Z E_t^\dagger}^2 = \normLtwoN{Y}^2 + \normLtwoN{Z}^2 \leq \normLtwo{\bm\xi}^2.
    \end{equation*}
    The estimate for $\mathcal L_{N,2}$ follows as in the proof of Proposition~\ref{prop: well-posedness}, using the estimates in Lemma~\ref{lem: norms}.
\end{proof}

\subsection{Consistency of the semidiscrete method}
% The norms  $\normLtwo{\cdot}$ and $\normLtwoN{\cdot}$ on $\CS$ and $\algebra$ are naturally extended to $(\CS)^2$ and $(\algebra)^2$. 
% On these product spaces we use the same notations for simplicity. 
First, lets prove the following lemma.
\begin{lemma}\label{lem: Lestimate}
    There exist $\alpha>0$ independent of $\omega_0, \hbar_N$ and $\bm \xi = \begin{bmatrix}
        f
        \\
        g
    \end{bmatrix} $, such that
    \begin{align*}
        \normLtwoN{ \Lambda_N p_N \bm \xi - p_N \mathcal{L} \bm \xi} \leq \alpha \hbar_N \normHfive{\omega_0} \normHfive{\bm \xi}
    \end{align*}
\end{lemma}

\begin{proof}
Using estimates from \autoref{sec: estimates}, we have
\begin{align*}
    &\normLtwoN{ \Lambda_N \begin{bmatrix}
        p_N f
        \\
        p_N g
    \end{bmatrix} - p_N \mathcal{L} \begin{bmatrix}
        f \\
        g
    \end{bmatrix}}^2 = 
    \\[2ex]
     \begin{split}
         &\quad  \normLtwoN{ \sbra{\Delta_N^{-1} p_N \omega_0 , p_N f}_N - p_N \pbra{\Delta^{-1}\omega_0,f} }^2
    \\
    & \, \quad + \normLtwoN{ \sbra{\Delta_N^{-1} p_N \omega_0, p_N g} + \sbra{ \Delta_N^{-1} p_N g , p_N \omega_0} - p_N\pbra{ \Delta^{-1}  \omega_0,  g} - p_N\pbra{ \Delta^{-1}  g, \omega_0}  }^2
     \end{split}
    \\[2ex]
    &\quad\leq \alpha^2 \hbar_N^2 ( \normHfive{\Delta^{-1}\omega_0}^2 \normHfive{f}^2 +  \normHfive{\Delta^{-1}\omega_0}^2 \normHfive{g}^2 +  \normHfive{\omega_0}^2 \normHfive{\Delta^{-1}g}^2 )
    \\
    &\quad\leq \alpha^2 \hbar_N^2  \normHfive{\omega_0}^2 \normHfive{\begin{bmatrix}
        f
        \\
        g
    \end{bmatrix}}^2
\end{align*}
where $\alpha>0$ is a constant.
\end{proof}

Then we can prove consistency, namely the following result.
\begin{proposition}\label{prop: consistency}
    For every $\bm \xi \in C^\infty(\mathbb{S}^2 \times \mathbb{S}^2)$ and $\lambda > c \coloneqq \sqrt{\frac{1}{4} + \frac{1}{2} \normHone{\omega_0}^2}$,
    \[
        \normLtwo{(I - \lambda \mathcal{L}_N)^{-1} \bm \xi - (I - \lambda \mathcal{L})^{-1} \bm \xi } \leq \frac{1}{\lambda - c} \hbar_N \parenthese{\alpha  \normHfive{\omega_0} \normHfive{\bm{\tilde\xi}}  + \sqrt{2}  \normHtwo{\mathcal{L} \bm{\tilde\xi}}}
        %\qquad\text{as}\qquad N\to \infty.
    \] 
    where $\tilde{\bm \xi} = (I - \lambda \mathcal{L})^{-1} \bm \xi$
    and $\alpha>0$ a constant as in Lemma~\ref{lem: Lestimate}.
\end{proposition}

\begin{proof}
% Let $c = \sqrt{\frac{1}{4} + \frac{1}{2} \normHone{\omega_0}^2}$. 
Recall from Proposition~\ref{prop: stability} that $\mathcal{L}_N$ is the infinitesimal generator of a $C_0$ semigroup $\mathcal{T}_N(t)$ satisfying $\normLtwo{\mathcal{T}_N(t)} \leq e^{ct}$.
The Hille--Yosida theorem (see, \textit{e.g.}, Pazy \cite[Thm. 1.5.3]{pazy_semigroups_1983}) then implies that 
\begin{equation*}
    \normLtwo{(I - \lambda \mathcal{L}_N)^{-1}} \leq \frac{1}{\lambda - c}.     
\end{equation*}
Thus we can proceed to the following estimates:
  \begin{align*}
        &\normLtwo{ (I - \lambda \mathcal{L}_N)^{-1} \bm \xi - (I - \lambda \mathcal{L})^{-1} \bm \xi } 
        \\
        &\qquad= 
        \normLtwo{(I - \lambda \mathcal{L}_N)^{-1} (I - \lambda \mathcal{L}_N) \parenthese{(I - \lambda \mathcal{L}_N)^{-1} \bm \xi - (I - \lambda \mathcal{L})^{-1} \bm \xi }}
        \\
        &\qquad \leq \frac{1}{\lambda - c} \normLtwo{ (I - \lambda \mathcal{L}_N) \parenthese{(I - \lambda \mathcal{L}_N)^{-1} \bm \xi - (I - \lambda \mathcal{L})^{-1} \bm \xi }}
        \\
        &\qquad \leq \frac{1}{\lambda - c} \normLtwo{ \bm \xi - (I - \lambda \mathcal{L}_N)(I - \lambda \mathcal{L})^{-1} \bm \xi }
        \\        
        &\qquad \leq \frac{1}{\lambda - c} \normLtwo{ (I - \lambda \mathcal{L})(I - \lambda \mathcal{L})^{-1}\bm \xi - (I - \lambda \mathcal{L}_N)(I - \lambda \mathcal{L})^{-1} \bm \xi }
        \\
       &\qquad \leq \frac{1}{\lambda - c} \normLtwo{  (\lambda \mathcal{L}_N - \lambda \mathcal{L}) (I - \lambda \mathcal{L})^{-1} \bm \xi }.
  \end{align*}

Let $\bm{\tilde\xi} = (I - \lambda \mathcal{L})^{-1} \bm \xi$. 
We then have from Lemma~\ref{lem: Lestimate} that
\begin{align*}
    \normLtwo{  ( \mathcal{L}_N - \mathcal{L})\bm{\tilde\xi} }  &\leq \underbrace{\normLtwo{  \pi_N( \mathcal{L}_N - \mathcal{L})\bm{\tilde\xi} }}_{\normLtwoN{ \Lambda_N \bm \xi - p_N \mathcal{L} \bm \xi}  } + \normLtwo{ (I - \pi_N)( \mathcal{L}_N - \mathcal{L})\bm{\tilde\xi} } 
    \\
    & \leq \alpha \hbar_N \normHfive{\omega_0} \normHfive{\bm{\tilde\xi}} + \normLtwo{ (I - \pi_N)( \mathcal{L}_N - \mathcal{L})\bm{\tilde\xi} } .
\end{align*}
Recall that $\mathcal{L}_N = \iota_N \Lambda_N p_N$, so $\mathcal{L}_N \bm{\tilde\xi}$ is a finite sum of spherical harmonics with $l \leq N-1$. Thus $\pi_N \mathcal{L}_N = \mathcal{L}_N$, and 
\begin{align*}
    \normLtwo{ (I - \pi_N)( \mathcal{L}_N - \mathcal{L})\bm{\tilde\xi} } = \normLtwo{ \pi_N \mathcal{L} \bm{\tilde\xi} - \mathcal{L}\bm{\tilde\xi} } .
\end{align*}
% \blu{Is the following valid?}
% {\color{red} KM: Yes, since $I-\lambda \mathcal L$ maps $C^\infty$ to $C^\infty$.}
As $\bm \xi \in C^\infty(\mathbb{S}^2 \times \mathbb{S}^2)$, we have 
% \begin{equation*}
%     \mathcal L \tilde{\bm \xi} = \mathcal{L}(I - \lambda \mathcal{L})^{-1} \bm \xi = (\mathcal{L}-I/\lambda)(I - \lambda \mathcal{L})^{-1} \bm \xi + 
% \end{equation*}
$\mathcal{L} \tilde{\bm \xi} = \mathcal{L}(I - \lambda \mathcal{L})^{-1} \bm \xi \in C^\infty(\mathbb{S}^2 \times \mathbb{S}^2) \subset H^2(\mathbb{S}^2 \times \mathbb{S}^2)$. Thus, by Lemma~\ref{prop: cv norms},
\begin{align*}
\normLtwo{ \pi_N \mathcal{L} \bm{\tilde\xi} - \mathcal{L}\bm{\tilde\xi} } \leq \sqrt{2} \hbar_N \normHtwo{\mathcal{L} \bm{\tilde\xi}}.
% \to 0\quad\text{as}\quad N\to \infty.
\end{align*}
% \blu{MP: could conclude more precisely by giving a convergence rate. By lemma \ref{prop: cv norms}, I have to show that $\mathcal{L}\bm{\tilde\xi} = \mathcal{L}(I - \lambda \mathcal{L})^{-1} \bm \xi$ is at least in $H^2$ to get a $O(\hbar_N)$ convergence. But this is always the case if $\bm \xi$ is $C^\infty$ ? }
\end{proof}

We now have everything to apply \autoref{thm: trotter-kato} above:
well-posedness from Proposition~\ref{prop: well-posedness};
stability from Proposition~\ref{prop: stability};
consistency from Proposition~\ref{prop: consistency}.
This concludes the proof of our second main result, stated in \autoref{th: eulerian and lagrangian} above.

\appendix

\section{Trotter and Kato theorem with convergence rate} \label{sec: proof trotter-kato}

In this section, we adapt the proof of the Trotter--Kato \autoref{thm: trotter-kato} from Pazy \cite{pazy_semigroups_1983}. 
Let us denote $\mathcal{R}_{\lambda} := (I - \lambda \mathcal{L})^{-1}$ and $\mathcal{R}_{\lambda,N} := (I - \lambda \mathcal{L}_N)^{-1}$. Fix $\bm \xi \in D(\mathcal{L})$ and an interval $[0,T]$. For any $t \in [0,T]$:
\begin{align*}
    \normLtwo{ (\mathcal{T}_N(t) - \mathcal{T}(t) ) \bm \xi } \leq & 
    \\
    & \underbrace{\normLtwo{ \mathcal{T}_N(t) (\mathcal{R}_{\lambda} - \mathcal{R}_{\lambda,N}) (\mathcal{R}_{\lambda}^{-1} \bm \xi) }}_{N_1}
 \\ +&\underbrace{\normLtwo{ \mathcal{R}_{\lambda,N} (\mathcal{T}_N(t) - \mathcal{T}(t) ) (\mathcal{R}_{\lambda}^{-1} \bm \xi) }}_{N_2} 
 \\
 +& \underbrace{\normLtwo{  (\mathcal{R}_{\lambda} - \mathcal{R}_{\lambda,N}) \mathcal{T}(t)(\mathcal{R}_{\lambda}^{-1} \bm \xi) }}_{N_3}
\end{align*}

Using the stability of $\mathcal{L}_N$ (Proposition~\ref{prop: stability}) and the consistency estimate (Proposition~\ref{prop: consistency}) we have
\begin{align*}
    N_1 \leq e^{cT} \frac{1}{\lambda - c} \hbar_N \parenthese{\alpha  \normHfive{\omega_0} \normHfive{ \bm \xi}  + \sqrt{2}  \normHtwo{\mathcal{L}  \bm \xi}} 
\end{align*}
and
\begin{align*}
    N_3 \leq  \frac{1}{\lambda - c} \hbar_N \parenthese{\alpha  \normHfive{\omega_0} \underset{t \in [0,T]}{\operatorname{sup}} \normHfive{ \mathcal{T}(t) \bm \xi}  + \sqrt{2}  \underset{t \in [0,T]}{\operatorname{sup}} \normHtwo{\mathcal{L}  \mathcal{T}(t) \bm \xi}} 
\end{align*}
where we have used the fact that $\mathcal{R}_{\lambda} \mathcal{T}(t) = \mathcal{T}(t) \mathcal{R}_{\lambda}$. 
For the term $N_2$, we first use the idendity (\textit{e.g.} Pazy, \cite[Lemma 3.4.1]{pazy_semigroups_1983})
\begin{align*}
 \mathcal{R}_{\lambda,N} (\mathcal{T}(t) - \mathcal{T}_N(t)) \mathcal{R}_{\lambda} \tilde{\bm \xi} = \int_0^t \mathcal{T}_N (t-s) \parenthese{\mathcal{R}_{\lambda} - \mathcal{R}_{\lambda,N}} \mathcal{T}(s) \tilde{\bm \xi}~ds.
\end{align*}
Thus we have
\begin{align*}
 &N_2 = \normLtwo{\mathcal{R}_{\lambda,N} (\mathcal{T}_N - \mathcal{T}(t) ) \mathcal{R}_{\lambda} (\mathcal{R}_{\lambda}^{-2} \bm \xi)}
 \\
 &\quad \leq\int_0^T \normLtwo{ \mathcal{T}_N (t-s) } \normLtwo{(\mathcal{R}_{\lambda} - \mathcal{R}_{\lambda,N}) \mathcal{T}(s)  (\mathcal{R}_{\lambda}^{-2} \bm \xi)}
 \\
 &\quad \leq T e^{cT} \frac{1}{\lambda - c} \hbar_N \left ( \alpha  \normHfive{\omega_0} \underset{t \in [0,T]}{\operatorname{sup}} \normHfive{ \mathcal{T}(t) (\mathcal{R}_{\lambda}^{-1} \bm \xi)} \right. 
 \\
 & \hspace{5cm} \left. + \sqrt{2}  \underset{t \in [0,T]}{\operatorname{sup}} \normHtwo{\mathcal{L}  \mathcal{T}(t) (\mathcal{R}_{\lambda}^{-1} \bm \xi)} \right).
\end{align*}

Then it follows from the above estimates that for any $\bm \xi \in D(\mathcal{L})$:
\begin{equation*}
        \normLtwo{\mathcal{T}_N(t) \bm \xi - \mathcal{T}(t) \bm \xi}\to 0 \quad\text{as}\quad N\to \infty.
    \end{equation*}
uniformly on $[0,T]$. Since the Hille-Yosida theorem implies that $D(\mathcal{L})$ is dense in $C^\infty(\mathbb{S}^2 \times \mathbb{S}^2)$, it follows that the previous statement holds for every $\bm \xi \in C^\infty(\mathbb{S}^2 \times \mathbb{S}^2)$. Moreover it is clear that regarding only the $N$ dependency, the convergence is $O(\hbar_N)$.

\arxivonly{

\section{Finite combination of spherical harmonics}\label{sec: finite combination}

Here, we focus the special case of stream functions that can be expressed as a \textit{finite sum} of spherical harmonics. Both numerator and denominator of sectional curvature are quadratic functions, so it is sufficient to compute it on any orthonormal basis. To do so, one can take advantage of the Lie algebra structure. 
Indeed if $(e_a)$ is an orthonormal basis of a Lie algebra $\mathfrak{g}$ with respect to a metric $\an{\cdot,\cdot}$, then the sectional curvature of the two-plane spanned by $e_a$ and $e_{a'}$ can be explicitly computed via the following formula \cite{milnor_curvatures_1976}
\begin{align*}\label{eq: Milnor formula}
    C(e_a,e_{a'}) = &\sum_{a''}  \frac{1}{2}f_{aa'a''}(-f_{aa'a''}+f_{a'a''a}+f_{a''aa'}) \\
    &-\frac{1}{4} (f_{aa'a''}-f_{a'a''a}+f_{a''aa'})(f_{aa'a''}+f_{a'a''a}-f_{a''aa'}) - f_{a''aa}f_{a''a'a'}
\end{align*}
where $f_{a a' a''} = \an{[e_{a},e_{a'}],e_{a''}}$ are structure constants with respect to the basis $(e_{a})$.
\\
For the infinite-dimensional algebra of stream functions $(C^\infty(\mathbb{S}^2),\an{\cdot,\cdot}_{H^{1}})$, an orthonormal basis is given by $(\tilde{\mathcal{Y}}_{lm} =  \mathcal{Y}_{lm} /\sqrt{l(l+1)})_{lm}$, where $\mathcal{Y}_{lm}$ are the standard spherical harmonics defined by \eqref{eq: spherical harmonics}. The structure constants $\tilde{f}_{lm,l'm',l''m''} = \an{\pbra{\tilde{\mathcal{Y}}_{lm} , \tilde{\mathcal{Y}}_{l'm'}} , \tilde{\mathcal{Y}}_{l''m''}}_{H^{1}}$ can be computed \cite{yoshida_riemannian_1997} as
\begin{align}
\begin{split}
  \tilde{f}_{lm,l'm',l''m''} =&  \sqrt{ \frac{  l''(l''+1) (2l+1)(2l'+1)(2l''+1)(l+l'-l'')(l+l'+l''+1) }{ (l+1) (l'+1)} }\\
&\times \frac{1-(-1)^{l+l'+l''}}{2} \frac{i}{\sqrt{4\pi}}\begin{bmatrix}
        l-1/2 & l'-1/2 & l'' \\
        1/2 & -1/2 & 0 \end{bmatrix} \begin{bmatrix}
        l & l' & l'' \\
        m & m' & -m'' \end{bmatrix}  .
\end{split}
\end{align}

For the corresponding matrix algebra $(\mathfrak{su}(N), \an{\cdot,\cdot}_{H^{1}_N})$, an orthonormal basis is given by $(\tilde{T}^N_{lm} =  T^N_{lm} / \sqrt{l(l+1)})_{lm}$, where $T^N_{lm}$ are the quantized harmonics defined in \eqref{eq: quantized hamronic}. And the structure constants $\tilde{f}^N_{lm,l'm',l''m''} = \an{\pbra{\tilde{T}^N_{lm} , \tilde{T}^N_{l'm'}} , \tilde{T}^N_{l''m''}}_{H^{1}_N}$ have been computed by Hoppe \cite{hoppe_quantum_1982}
\begin{align}\begin{split}
    &\tilde{f}^N_{lm,l'm',l''m''} = \\ & (1-(-1)^{l+l'+l''})(-1)^{m''+1}  \frac{l'' (l''+1)}{\sqrt{l(l+1) l'(l'+1)}} \sqrt{2l+1}\sqrt{2l'+1}\sqrt{2l''+1} 
    \\
    &\times
    \begin{bmatrix}
        l & l' & l'' \\
        m & m' & m''
    \end{bmatrix}
    \begin{Bmatrix}
        l & l' & l'' \\
        \frac{N-1}{2} & \frac{N-1}{2} & \frac{N-1}{2}
    \end{Bmatrix}
    \end{split}
\end{align}
where $\{:::\}$ denotes the Wigner $6j$ symbol.

In his PhD thesis, Hoppe proved the pointwise convergence of structure constants (\cite{hoppe_quantum_1982}, I.B.II)
\begin{align}
    \tilde{f}^N_{lm,l'm',l''m''} = \tilde{f}_{lm,l'm',l''m''} + O\parenthese{\frac{1}{N^2}}.
\end{align}
% Using Milnor's formula \eqref{eq: Milnor formula}, we have $C(p_N \mathcal{Y}_{lm}, p_N \mathcal{Y}_{l'm'}) \underset{N \rightarrow \infty}{\longrightarrow} C( \mathcal{Y}_{lm}, \mathcal{Y}_{l'm'})$, which proves the following proposition
Then using Milnor's formula, the following proposition holds.
\begin{proposition}
    Let $f,g \in C^\infty(\mathbb{S}^2)$ that can be expressed as \textit{finite} combinations of spherical harmonics. Then
    \begin{align}
        C_N(p_N f, p_N g) \rightarrow C( X_f, X_g) \quad \textit{as} \quad N \rightarrow \infty.
    \end{align}
\end{proposition}

}

% \printbibliography
\bibliographystyle{amsplainnat}
\bibliography{references2}

\end{document}